\documentclass[preprint,12pt]{elsarticle}

\usepackage{amssymb}
\usepackage{amsthm}
\usepackage{amsmath}

\usepackage{nccbbb}
\usepackage{hyperref}
\usepackage{cases}

\newdefinition{definition}{Definition}[section]
\newdefinition{example}[definition]{Example}
\newdefinition{remark}[definition]{Remark}
\newtheorem{lemma}[definition]{Lemma}
\newtheorem{theorem}[definition]{Theorem}
\newtheorem{corollary}[definition]{Corollary}

\newtheorem{property}[definition]{Property}

\newcommand{\E}{\mathcal{E}} 
\newcommand{\B}{\mathcal{B}} 
\newcommand{\F}{\mathcal{F}} 
\newcommand{\A}{\mathcal{A}} 
\newcommand{\D}{\mathcal{D}} 
\newcommand{\R}{\mathbb{R}} 
\newcommand{\N}{\mathbb{N}} 
\newcommand{\Prb}{\mathbb{P}} 
\newcommand{\Ep}{\mathbb{E}} 
\newcommand{\ud}{\mathrm{d}} 
\newcommand{\slog}{\mathrm{slog}\,} 
\newcommand{\sexp}{\mathrm{sexp}\,} 

\numberwithin{equation}{section}

\journal{XXX}

\begin{document}

\begin{frontmatter}



\title{Measure-Valued Generators of General Piecewise Deterministic Markov Processes\tnoteref{t1}}
\tnotetext[t1]{This work was supported by National Natural Science Foundation of China (11471218) and Hebei Higher School Science and Technology Research Projects (ZD20131017)}


\author[csu]{Zhaoyang Liu}
\ead{liu.zhy@csu.edu.cn}

\author[csu]{Yong Jiao}
\ead{jiaoyong@csu.edu.cn}

\author[stdu]{Guoxin Liu\corref{cor1}}
\ead{liugx@stdu.edu.cn}

\cortext[cor1]{Corresponding author}

\address[csu]{School of Mathematics and Statistics, Central South University, Changsha, Hunan, China, 410083}
\address[stdu]{Department of Applied Mathematics and Physics, Shijiazhuang Tiedao University, Shijiazhuang, Hebei, China, 050043}

\begin{abstract}
  We consider a piecewise-deterministic Markov process (PDMP) with general conditional distribution of inter-occurrence time, which is called a general PDMP here. Our purpose is to establish the theory of measure-valued generator for general PDMPs. The additive functional of a semi-dynamic system (SDS) is introduced firstly, which presents us an analytic tool for the whole paper. The additive functionals of a general PDMP are represented in terms of additive functionals of the SDS. The necessary and sufficient conditions of being a local martingale or a special semimartingale for them are given. The measure-valued generator for a general PDMP is introduced, which takes value in the space of additive functionals of the SDS. And its domain is completely described by analytic conditions. The domain is extended to the locally (path-)finite variation functions. As an application of measure-valued generator, we study the expected cumulative discounted value of an additive functional of the general PDMP, and get a measure integro-differential equation satisfied by the expected cumulative discounted value function.
\end{abstract}

\begin{keyword}
piecewise-deterministic Markov process \sep semi-dynamic system \sep additive functional \sep measure-valued generator \sep measure integro-differential equation



\end{keyword}

\end{frontmatter}


\section{Intruduction}

Piecewise deterministic Markov processes (PDMPs in short) are a general class of Markov processes, for which the randomness is only on the jumping times and post-jump locations. Between two adjacent random jumps, a PDMP evolves as a semi-dynamic system (SDS in short). The pioneering work of PDMPs is Davis \cite{davis1984piecewise}, which introduces the definition and present the theory of PDMPs. Since then PDMPs have attracted wide attention from many fields, in particular as far as application in finance, insurance, statistical physics and queuing are concerned. For applications in insurance see in particular Dassios and Embretchts \cite{Dassios1989Martingales} and Rolski et al \cite{rolski1998stochastic}. They point out that PDMPs theory provides a systematic toolet for the study of the ruin theory. Kirch and Runggaldier \cite{Kirch2004Efficient} use the reduction technique to solve a hedging problem in a continuous-time market (with uncontrolled draft). Dai \cite{dai1995on} develops a unified approach via fluid limit model by virtue of PDMPs in queueing networks. Further applications in queuing can be found in B\"auerle \cite{Bauerle2001Discounted}, Graham and Robert \cite{Graham2009Interacting} and Rieder and Winter \cite{Rieder2009Optimal}. Applications in statistical physics see Faggionato et al. \cite{Faggionato2009Non} for example.

PDMPs theory and their optimization have been developed deeply in the last two decades. Davis \cite{davis1993markov} not only gives the framework of the model, but also, based on the local martingale theory of jumping processes, transplants the extended generator proposed by Stroock and Varadhan \cite{stroock1979multidimensional} for multidimensional diffusion processes into PDMPs theory and get the corresponding It\^o formula. Meanwhile, he develops the optimal control theory for PDMPs also. Jacod and Skorokhod \cite{jacod1996jumping} generalize the concept of PDMPs to general cases, in which they call them jumping Markov processes (JMPs in short). They study the characteristics of the general PDMPs and the normal subjects of Markov processes, such as additive functionals, semimartingales and semimartingale functions. Palmowski and Rolski \cite{palmowski2002technique} obtain the general expression of exponential martingale with the extended generator for a wide range of Markov processes, including Davis' PDMPs as special cases, and set up the corresponding theory of change of measure. The stability and ergodicity of PDMPs are developed by Costa \cite{costa1990stationary}, Costa and Dufour \cite{costa1999stability,costa2008stability}, Jin and Amin \cite{Jin2016Stability}. Based on the theory of marked point processes, Jacobsen \cite{jacobsen2006point} represents not only the time-homogeneous PDMPs, but also the nonhomogeneous ones. Hou and Liu \cite{hou2005markov} discuss the analytic properties of the characteristic triple of the general PDMPs and develop the concept of extended generator with a discrete part. The comprehensive literature on the stochastic control of PDMPs in insurance can be found in Schmidli \cite{schmidli2008stochastic}, Azcue and Muller \cite{azcue2014stochastic} and the references therein. Continuous average control of PDMPs is developed by Costa and Dufour \cite{costa2013continuous} and the references therein. de Saporta et al. \cite{deSaporta2016Optimal} study the impulse control of PDMPs. The numerical methods for stochastic control of PDMPs are studied by Brandejsky et al. \cite{Brandejsky2012Numerical}, de Saporta and Dufour \cite{deSaporta2010Numerical-impulse}, de Saporta et al. \cite{deSaporta2010Numerical-stopping,deSaporta2015Numerical} and the references therein.

The theory of infinitesimal characteristics of Markov processes is always one of the core research problems of Markov processes. Kolmogorov obtained the Kolmogorov forward and backward equations when he was investigating a class of Markov processes with finite-dimensional Euclidean state space whose transition probabilities possess density functions. Meanwhile, he found that the processes are characterized by the coefficients of the corresponding equations which have simple probabilistic meanings and which are infinitesimal characteristics of the processes. The infinitesimal characteristics of the processes introduced by Kolmogorov allow us to determine not only the transition probabilities but also to evaluate the distributions of various functionals of the processes. Utilizing the theory of semigroups of operators, Feller pointed out that this semigroup completely determines the transition probabilities, and that the semigroup is in many cases uniquely determined by its infinitesimal generator. Feller suggested that the infinitesimal generator be considered as an infinitesimal characteristic of the process. Kolmogorov's ideas constituted the basis of the mathematical theory of Markov processes and provided the direction for further investigations.

A PDMP is called to be \emph{quasi-It\^o}, in the terminology of \cite{jacod1996jumping}, if the predictable brackets of all the (locally) square-integrable martingales are absolutely continuous with respect to the Lebesgue measure. Jacod and Skorokhod \cite{jacod1996jumping} pointed out that it is equivalent to that all the conditional jumping time distributions have densities.
Jacod and Skorokhod \cite{jacod1996jumping} discuss the (weak) infinitesimal generators and extended generators. Jacobsen \cite{jacobsen2006point} extend the extended generators to the nonhomogeneous cases.
The situations they deal with for extended generators are both quasi-It\^o, which leads to that the domains of the extended generators are limited in the class of absolutely (path-)continuous functions (see \cite{jacobsen2006point}). Hou and Liu \cite{hou2005markov} introduce a new generator for the general PDMP, the extended generator coupled with a discrete part, which leads to the situation they deal with beyond the quasi-It\^o case, for which there is no singularly continuous part in the conditional jumping time distributions. This makes the domain extended to the class of piecewise absolutely (path-)continuous functions instead of the one of absolutely (path-)continuous functions. Just as Jacobsen \cite{jacobsen2006point} pointed out, it is of course advantageous to have the domain of the generator as large as possible.

The purpose of this paper is to establish the theory of the so-called measure-valued generator for general PDMPs. To do this, we firstly introduce the definition of the additive functionals of an SDS. Some basic properties of the additive functionals are discussed. Especially, we give the characteristics of the absolutely continuous part and purely discontinuous part in the Lebesgue decomposition of an additive functional. The concept of additive functionals of SDSs is the analytic basis for the general PDMPs theory.
Just like Jacod and Skorokhod said, the case of general PDMPs is much more technical. Based on the properties of the additive functionals of the SDS, the technical obstacle is overcome. We provide the representation theorem in terms of the additive functionals of the SDS and the necessary and sufficient conditions of being local martingales and special semimartingales for the additive functionals of the general PDMP. Therefore, we extend the results about additive functionals of \emph{quasi-Hunt} JMPs in \cite{jacod1996jumping} to general PDMPs.
Then, by virtue of the analytic theory of the additive functionals of SDSs and the ideas about extended generator, we introduce the concept of measure-valued generator which is valued in the space of additive functionals of the SDS with locally finite variation. We also give the description of its domain. Roughly speaking, the domain of generator is extended from absolutely (path-)continuous functions to functions with locally (path-)finite variation. And the corresponding It\^o formula is followed by the way. As the direct application of measure-valued generator, the domain of the extended generator is extended further. For a PDMP in the sense of Davis \cite{davis1993markov}, the results in this paper are almost the same except for a slight modification on his boundary condition. By the way, we also discuss the L-extended generator proposed by Kunita \cite{Kunita1969Absolute} for a general PDMP, which generalizes the corresponding results of Jacod and Skorokhod \cite{jacod1996jumping}.
Finally, as an application of the theory of measure-valued generator, we get a new measure integro-differential equation which is satisfied by the expected cumulative discounted value of an additive functional of a general PDMP, and give the uniqueness condition of the solution to the measure integro-differential equation. For illustration, we point out that, in some certain cases, the measure integro-differential equation becomes an integro-differential equation or an impulse integro-differential equation.

This paper is organized as follow.
We start with the definition of general piecewise deterministic Markov processes, and get their characteristic triple $(\phi,F,q)$ in Section 2. Roughly speaking, a general PDMP is a jumping Markov process introduced by Jacod and Skorokhod \cite{jacod1996jumping}, i.e., a strong Markov process with natural filtration of discrete type. Different from Jacod and Skorokhod \cite{jacod1996jumping}, we introduce the general PDMPs in the way of Hou and Liu \cite{hou2005markov}. By virtue of the concept of piecewise deterministic Markov skeleton processes, we get the necessary and sufficient conditions for such a process to be a strong Markov process in Theorem \ref{thm.PDP=>str.Markov}, and present the characterization of characteristic triple $(\phi,F,q)$ of a general PDMP, which generalize the corresponding results for quasi-Hunt cases in Jacod and Skorokhod \cite{jacod1996jumping}. Based on these, some basic properties of the characteristic triple $(\phi,F,q)$ are studied separately one by one.

Motivated by the representation theorem (Theorem 15 in \cite{jacod1996jumping}) for additive functionals of quasi-Hunt JMPs, we introduce the so-called additive functionals of a semi-dynamic system in Section 3. Some basic examples and properties are presented. Specially, the representations are given for the absolutely continuous part and the purely discontinuous part of an additive functional with locally finite variation, which depend only on the current state along the path of the SDS $\phi$. The concept of additive functionals of SDSs plays a key role as an analytic tool in this paper.
In Section 4, we devote to generalize the results of \cite{jacod1996jumping} for the additive functionals of quasi-Hunt JMPs to the general PDMPs. Theorem \ref{thm.A=a+b} shows that an additive functional of a general PDMP can be represented in terms of an additive functional $a$ of  the SDS $\phi$ and a measurable function $b:E\times\R_+\times E\mapsto\R$ with property (\ref{eq.b.invariant}). It is different from Theorem 15 of \cite{jacod1996jumping} that here $b$ is satisfied (\ref{eq.b.invariant}) instead of (\ref{eq.b=bar.b}). Furthermore, we present the necessary and sufficient conditions for an additive functional of a general PDMP to be a local martingale or a special semimartingale in terms of the corresponding $a$ and $b$.

In Section 5, we come to define the so-called measure-valued generators of general PDMPs. The analytic formulation is presented for the measure-valued generator, and its domain is completely described in an analytic way. By the way, we generalize the extended generators and L-extended generators for general PDMPs. Form which we can see that the measure-valued generator is a natural generalization of both the extended generator and L-extended generator for general PDMPs. As corollaries, we get the corresponding results for quasi-Hunt, quasi-It\^o and quasi-step cases respectively. It shows that the concept of the extended generators is more suitable to the quasi-It\^o PDMPs than the general cases.
As an application of the measure-valued generator, we study the expected cumulated discounted value function of an additive functional of a general PDMP in Section 6. Due to the extended domain of the measure-valued generator, the expected value function with locally (path-)finite variation can be investigated. It yields a measure integro-differential equation in terms of the measure-valued generator in general (see Theorem \ref{thm.V.expectation.predictable} and Corollary \ref{cor.V.expectation.optional}). And, under some certain conditions, the solution of the measure integro-differential equation is exactly the expected value function (see Theorem \ref{thm.UniqueSolution}). In the quasi-It\^o case, the expected value function of the usual integral type functional (\ref{eq.V.expectation.Ito}) is surely absolutely (path-)continuous and satisfies the usual integro-differential equation in terms of the extended generator but in the sense of (path-)almost everywhere (see Corollary \ref{cor.V.expectation.Ito} and \ref{cor.UniqueSolution.Ito}), which extend the results of Theorem 32.3 and 32.10 of Davis \cite{davis1993markov}. Especially, in some other special PDMPs, the expected value function may satisfy some integro-differential equations or impulse integro-differential equations (see Corollary \ref{cor.V.expectation.step}-\ref{cor.UniqueSolution.nonsingular}).

\section{Definitions and properties}

\subsection{Definition of general PDMPs}

Let $(\Omega, \F, \Prb)$ be a complete probability space, $E$ be a Polish space, and $\E$ the Borel $\sigma$-algebra on $E$. And let $X=\{X_t\}_{0\leqslant t<\tau}$ be a c\`adl\`ag stochastic process with lifetime $\tau$ defined on $(\Omega, \F, \Prb)$  and taking values in $(E, \E)$, and $\{\F_t\}_{0\leqslant t<\tau}$ the natural filtration of $X$.

In order to introduce the general PDMPs, we need the concept of piecewise deterministic Markov skeleton processes (PDPs in short) introduced by Hou and Liu \cite{hou2005markov}.

\begin{definition}
  A c\`adl\`ag stochastic process $X=\{X_t\}_{0\leqslant t<\tau}$ is called a \emph{piecewise deterministic Markov skeleton process} (PDP) if there exists a strictly increasing sequence of nonnegative random variables $\{\tau_n\}_{n\geqslant 0}$ (i.e., for any $n\geqslant 1$, $\tau_n<\tau_{n+1}$ if $\tau_n<\tau$) with $\tau_0=0$, $\tau_n\uparrow\tau$, and a sequence of measurable mappings $\phi_n:\R_+\times E\mapsto E$ such that
  \begin{equation}\label{eq.PiecewiseDeterministic}
    X_t=\sum_{n=0}^\infty\phi_n(t-\tau_n,X_{\tau_n})\bbbone_{\{\tau_n\leqslant t<\tau_{n+1}\}},
    \quad 0\leqslant t<\tau \quad \hbox{a.s.};
  \end{equation}
  and for any $n\in\N$ and a bounded measurable function $f$ defined on $(E^{[0,\infty)},$ $\E^{[0,\infty)})$,
  \begin{equation}\label{eq.MarkovSkeletonProperty}
    \Ep\big[f(X_{\tau_n+\cdot})|\F_{\tau_n}\big]=
    \Ep\big[f(X_{\tau_n+\cdot})|X_{\tau_n}\big].
  \end{equation}
\end{definition}

The formula (\ref{eq.PiecewiseDeterministic}) describes that the process is piecewise deterministic. The property (\ref{eq.MarkovSkeletonProperty}) states that the process has Markov property at each time $\tau_n$, which is called \emph{Markov skeleton property} as in Hou and Liu \cite{hou2005markov}. It is easily to see that the natural filtration $\F$ of the process $X$ is a \emph{discrete type filtration} (or in other word, a \emph{jumping filtration}), i.e.,
\begin{equation*}
  \F_t=\bigcup_{n=0}^\infty\big(\F_{\tau_n}\cap\{\tau_n\leqslant t<\tau_{n+1}\}\big),\quad 0\leqslant t<\tau,
\end{equation*}
where $\F_{\tau_n}=\sigma(X_0,\tau_1,X_{\tau_1},\dots,\tau_n,X_{\tau_n})$ (See Hou and Liu \cite{hou2005markov}, or Jacod and Skorohod \cite{Jacod1994Jumping} for examples). Let $\F_\infty=\bigvee_{n=1}^{\infty}\F_{\tau_n}$. Suppose that there exists a family of probability measures $\Prb_x$ on $(\Omega, \F)$ for $x\in E$ satisfying that for any fixed $B\in \F_{\infty}$, $x\mapsto \Prb_x\{B\}$ is $\E$-measurable, and for any fixed $x\in E$,
\begin{equation*}
  \Prb_x\{B\}=\Prb\{B|X_0=x\}, \quad B\in\F_{\infty}.
\end{equation*}

For a given PDP $X=\{X_t\}_{0\leqslant t<\tau}$, set $\sigma_0=0$, $\sigma_n:=\tau_n-\tau_{n-1}\,(n\geqslant 1)$. The Markov skeleton property (\ref{eq.MarkovSkeletonProperty}) implies that the sequence $\{(\sigma_n,X_{\tau_n})\}$ is a Markov sequence taking values in Polish space $(\R_+\times E,\B(\R_+)\times\E)$, and its transition kernel is independent of the first component. We call it a \emph{Markov skeleton sequence}. The transition kernel of the sequence $\{(\sigma_n,X_{\tau_n})\}$ is denoted by $\{G_n(x,\ud t,\ud y)\}_{n\geqslant 0}$.
Let $F_n(x,\ud t):=G_n(x,\ud t,E)$. Then for any $B\in\E$, $G_n(x,\ud t,B)\ll F_n(x,\ud t)$. Thus, by Radon-Nikodym Theorem, there exists a sequence of  $q_n:E\times\R_+\times\E\mapsto[0,1]$ such that $q_n(x,t,\cdot)$ is a probability measure on $(E,\E)$ for any fixed $(x,t)$, and $q_n(\cdot,\cdot,B)$ is $\E\times\B(\R_+)$-measurable for any fixed $B\in\E$.
$\{(\phi_n,F_n,q_n)\}$ is called the \emph{characteristic triple sequence} of the PDP $X$.
A PDP $X$ is called to be \emph{homogeneous} if the characteristic triple sequence is independent of $n$, i.e., $(\phi_n,F_n,q_n)\equiv(\phi,F,q)$. The triple $(\phi,F,q)$ is called the \emph{characteristic triple} of the homogeneous PDP $X$.

For a homogeneous PDP $X$ with the characteristic triple $(\phi,F,q)$, let $F(x,t):=F(x,(t,\infty])$ for $x\in E$, which is called the \emph{conditional survival function}, and $c(x):=\inf\{t>0:F(x,t)=0\}$ ($\inf\emptyset=+\infty$ by convention). Denote
$$\mathcal{I}_x:=
\left\{
  \begin{array}{ll}
    \R_+, & c(x)=\infty; \\
    \displaystyle[0,c(x)), & c(x)<\infty,\,F(x,c(x)-)=0; \\
    \displaystyle[0,c(x)], & c(x)<\infty,\,F(x,c(x)-)>0.
  \end{array}
\right.
$$

Now we present the definition of general PDMPs.
\begin{definition}
  A homogeneous PDP $X=\{X_t\}_{0\leqslant t<\tau}$ is called a \emph{general PDMP} if it is a strong Markov process.
\end{definition}

With the definitions above, we have the following theorem.

\begin{theorem}\label{thm.PDP=>str.Markov}
  A homogeneous PDP $X=\{X_t\}_{0\leqslant t<\tau}$ with characteristic triple $(\phi,F,q)$ is a homogeneous strong Markov process if and only if for any $x\in E$, $s,t\in\R_+$ with $s+t\in\mathcal{I}_x$ and $B\in\E$, the characteristic triple $(\phi,F,q)$ satisfies
  \textnormal{\begin{enumerate}[(i)]
    \item $\phi(0,x)=x$, $\phi(t,\phi(s,x))=\phi(s+t,x)$;
    \item $F(x,s)F(\phi(s,x),t)=F(x,s+t)$;
    \item $q(x,s+t,B)=q(\phi(s,x),t,B)$.
  \end{enumerate}}
\end{theorem}

\begin{proof}
  Now we will come to prove the sufficiency. First we will prove the simple Markov property, that is to prove for any $s,t\in\R_+$,
  $$\Ep_x[f(X_{t+s})|\F_t]=\Ep[f(X_{t+s})|X_t], \quad x\in E$$
  for each bounded measurable function $f$. By the definition of a homogeneous PDP and the property of the discrete type filtration, we have
  $$\Prb\{\sigma_{n+1}>s\,|\F_{\tau_n}\}=\Prb\{\sigma_{n+1}>s\,|X_{\tau_n}\}=F(X_{\tau_n},s).$$
  Then by (ii), we have
  \begin{align*}
      & \Prb_x\{\tau_{n+1}>t+s\,|\F_t\}\bbbone_{\{\tau_n\leqslant t<\tau_{n+1}\}} \\
    = & \frac{F(X_{\tau_n},t+s-\tau_n)}{F(X_{\tau_n},t-\tau_n)}
    \bbbone_{\{\tau_n\leqslant t<\tau_{n+1}\}}\\
    = & \frac{F(X_{\tau_n},t-\tau_n)\,F(\phi(t-\tau_n,X_{\tau_n}),s)}{F(X_{\tau_n},t-\tau_n)} \bbbone_{\{\tau_n\leqslant t<\tau_{n+1}\}}\\
    = & F(\phi(t-\tau_n,X_{\tau_n}),s)\,\bbbone_{\{\tau_n\leqslant t<\tau_{n+1}\}}\\
    = & F(X_t,s)\,\bbbone_{\{\tau_n\leqslant t<\tau_{n+1}\}}. 
  \end{align*}
  Denote $\tilde{\tau}:=\tau_{n+1}$ if $\tau_n\leqslant t<\tau_{n+1}$, i.e., $\tilde{\tau}$ is the next random jumping time. Summing up the both sides, we have
  \begin{equation}\label{eq.MarkovProperty.1}
    \Prb_x\{\tilde{\tau}>t+s\,|\F_t\}=F(X_t,s), \quad t<\tau.
  \end{equation}
  Notice that
  \begin{gather*}
    F(x,t+\ud s)=F(x,t)\,F(\phi(t,x),\ud s), \\
    q(x,t+s,B)=q(\phi(t,x),s,B).
  \end{gather*}
  And by the definition of a homogeneous PDP and the property of the discrete type filtration, we have
  \begin{align*}
     & \Prb_x\{X_{\tilde{\tau}}\in B\,|\F_t\}\,\bbbone_{\{\tau_n\leqslant t<\tau_{n+1}\}} \\
    = & \Prb_x\{X_{\tau_{n+1}}\in B\,|\F_t\}\,\bbbone_{\{\tau_n\leqslant t<\tau_{n+1}\}} \\
    = & \frac{\int_{(t-\tau_n,\infty)}q(X_{\tau_n},u, B)\,F(X_{\tau_n},\ud u)}{F(X_{\tau_n},t-\tau_n)} \bbbone_{\{\tau_n\leqslant t<\tau_{n+1}\}} \\
    = & \int_{(0,\infty)}q(\phi(t-\tau_n,X_{\tau_n}),s,B)F(\phi(t-\tau_n,X_{\tau_n}),\ud s)
    \bbbone_{\{\tau_n\leqslant t<\tau_{n+1}\}} \\
    = & \int_{(0,\infty)}q(X_t,s,B)F(X_t,\ud s)\,
    \bbbone_{\{\tau_n\leqslant t<\tau_{n+1}\}} \\
    = & \Prb_x\{X_{\tilde{\tau}}\in B\,|X_t\}\,\bbbone_{\{\tau_n\leqslant t<\tau_{n+1}\}}. 
  \end{align*}
  Again, sum up the both sides,
  \begin{equation}\label{eq.MarkovProperty.2}
    \Prb_x\{X_{\tilde{\tau}}\in B\,|\F_t\}=\Prb_x\{X_{\tilde{\tau}}\in B\,|X_t\}, \quad t<\tau.
  \end{equation}
  for all $x\in E$.

  From (\ref{eq.MarkovProperty.1}) and (\ref{eq.MarkovProperty.2}), the distributions of $\tilde{\tau}$ and $X_{\tilde{\tau}}$ conditioning on $\F_t$ both only depend on $X_t$. And any $\sigma_n$, the time length between two adjacent random jumps after $\tilde{\tau}$, and each post-jump state are only relate to $X_{\tilde{\tau}}$. While, $X_{\tilde{\tau}}$ is only dependent on $X_t$ and independent on the history before $t$. So we prove that if $f:(E^{[0,\infty)},\E^{[0,\infty)})\mapsto(\R,\B(\R))$ is bounded and measurable, then
  $$\Ep_x[f(X_{t+\cdot})|\F_t]=\Ep_x[f(X_{t+\cdot})|X_t]$$
  for each $x\in E$.

  In order to get the strong Markov property, we use the character of stopping time of discrete type filtration. If we let $Y_n:=(X_0,\sigma_1,X_{\tau_1},\dots,\sigma_n,X_{\tau_n})$, then for each stopping time $T$, there exists a sequence of functions $s_1,s_2,\dots$ such that
  $$T\bbbone_{\{\tau_{n-1}<T\leqslant\tau_n\}}=\big[\big(\tau_{n-1}+s_n(Y_{n-1})\big)\land\tau_n\big] \bbbone_{\{\tau_{n-1}<T\leqslant\tau_n\}}.$$
  This means that there are three cases on $\{T<\infty\}$: $T=0$, $T=\tau_n$ for some $n$, or $T=\tau_{n-1}+s_n(Y_{n-1})$ for some $n$. Furthermore, $\{T=0\}$, $\{T=\tau_n\}$ and $\{T=\tau_{n-1}+s_n(Y_{n-1})\}$ are all $\F_T$-measurable. Likewise, we define that $\tilde{\tau}:=\tau_n$ if $\tau_{n-1}\leqslant T<\tau_n$. Making the same deduction as above for the three cases we can show that
  \begin{gather*}
    \Prb_x\{(\tilde{\tau}>T+s)\cap(T<\infty)|\F_T\}=F(X_T,s)\bbbone_{\{T<\infty\}}, \\
    \Prb_x\{X_{\tilde{\tau}}\in\ud y\,|\F_T\}=\Prb_x\{X_{\tilde{\tau}}\in\ud y\,|X_T\}.
  \end{gather*}
  The same reasoning as the proof of Markov property can prove that, for any bounded and measurable function $f:(E^{[0,\infty)},\E^{[0,\infty)})\mapsto(\R,\B(\R))$, we have
  $$\Ep_x[f(X_{T+\cdot})\bbbone_{\{T<\infty\}}|\F_T]=
  \Ep_x[f(X_{T+\cdot})|X_T]\bbbone_{\{T<\infty\}}.$$
  Hence we proved the strong Markov property of the process $X$.

  Now we are in the position to give the proof of necessity.

  We begin with the property (ii). For each $x\in E$, $t\in\mathcal{I}_x$ and $s\geqslant 0$, we have
  \begin{align*}
    F(x, t+s) = & \Prb_x\{\tau_1>t+s\}\\
  = & \Prb_x\{\tau_1>t\}\,\Prb_x\{\tau_1>t+s|\tau_1>t\}\\
  = & F(x,t)\,\Ep_x[\Prb_x\{\tau_1>t+s|\F_t\}|\tau_1>t]\\
  = & F(x,t)\,\Ep_x[\Prb_x\{\tau_1>t+s|X_t\}|\tau_1>t]\\
  = & F(x,t)\,\Ep_x[\Prb_{\phi(t,x)}\{\tau_1>t+s\}|\tau_1>t]\\
  = & F(x,t)\,\Prb_{\phi(t,x)}\{\tau_1>t+s\}\\
  = & F(x,t)F(\phi(t,x),s).
  \end{align*}
  The property (ii) is proved.

  We turn to the proof of the property (i). For $x\in E$, $s,t\in \mathcal{I}_x$ with $t+s\in\mathcal{I}_x$, and a single point set $\{\phi(t+s,x)\}\in \E$, we consider the conditional expectation $\Ep_x[\bbbone_{\{\phi(t+s,x)\}}(X_{t+s})\,\bbbone_{\{\tau_1>t+s\}}|\F_t]$. On the one hand, It follows from Markov property that
  \begin{align*}
    &\Ep_x[\Ep_x[\bbbone_{\{\phi(t+s,x)\}}(X_{t+s})\,\bbbone_{\{\tau_1>t+s\}}|\F_t]]\\
   =&\Ep_x[\Ep_x[\bbbone_{\{\phi(t+s,x)\}}(X_{s}\circ \theta_t)\,\bbbone_{\{\tau_1>t\}}\,\bbbone_{\{\tau_1\circ\theta_t>s\}}|\F_t]]\\
   =&\Ep_x[\bbbone_{\{\tau_1>t\}}]\,\Ep_x[\Ep_x[\bbbone_{\{\phi(t+s,x)\}}(X_{s}\circ \theta_t)\,\bbbone_{\{\tau_1\circ\theta_t>s\}}|X_t]]\\
   =&F(x,t)\, \Ep_{\phi(t,x)}[\bbbone_{\{\phi(t+s,x)\}}(\phi(s,\phi(t,x)))\,\bbbone_{\{\tau_1>s\}}]\\
   =&\bbbone_{\{\phi(t+s,x)\}}(\phi(s,\phi(t,x))\,F(x,t)\,F(\phi(t,x),s),
  \end{align*}
  where $\theta_t$ is the shift operator.
  On the other hand,
  \begin{align*}
    &\Ep_x[\Ep_x[\bbbone_{\{\phi(t+s,x)\}}(X_{t+s})\,\bbbone_{\{\tau_1>t+s\}}|\F_t]]\\
   =&\Ep_x[\bbbone_{\{\phi(t+s,x)\}}(X_{t+s})\,\bbbone_{\{\tau_1>t+s\}}]\\
   =&\Ep_x[\bbbone_{\{\phi(t+s,x)\}}(\phi(t+s,x))\,\bbbone_{\{\tau_1>t+s\}}]\\
   =&\Ep_x[\bbbone_{\{\tau_1>t+s\}}]\\
   =&F(x,t+s).
  \end{align*}
  Therefore,
  $$\bbbone_{\{\phi(t+s,x)\}}(\phi(s,\phi(t,x))\,F(x,t)\,F(\phi(t,x),s)=F(x,t+s).$$
  In these cases, $F(x,t+s)>0$. It follows from (ii) that
  $$\bbbone_{\{\phi(t+s,x)\}}(\phi(s,\phi(t,x)))=1.$$
  That is $\phi(t+s,x)=\phi(s,\phi(t,x))$. It is obviously that $\phi(0,x)=x$.

  At last, for $x\in E$, $s,t\in \mathcal{I} _x$ with $t+s\in\mathcal{I} _x$ and $B\in \E$, we have
  \begin{align*}
    G(x, t+\ud s, B) = & \Prb_x\{\tau_1\in t+\ud s, X_{\tau_1}\in B\}\\
   = & F(x,t)\,\Prb_x\{\tau_1\in t+\ud s, X_{\tau_1}\in B|\tau_1>t\}\\
   = & F(x,t)\,\Ep_x[\Prb_x\{\tau_1\in t+\ud s, X_{\tau_1}\in B|\F_t\}|\tau_1>t]\\
   = & F(x,t)\,\Ep_x[\Prb_{\phi(t,x)}\{\tau_1\in t+\ud s, X_{\tau_1}\in B\}|\tau_1>t]\\
   = & F(x,t)\,\Prb_{\phi(t,x)}\{\tau_1\in t+\ud s, X_{\tau_1}\in B\}\\
   = & F(x,t)\,G(\phi(t,x), \ud s, B).
  \end{align*}
  Hence,
  \begin{align*}
    q(x, t+s, B) = & \frac{G(x,t+\ud s, B)}{F(x, t+\ud s)}\\
   = & \frac{F(x,t)G(\phi(t,x),\ud s, B)}{F(x,t)F(\phi(t,x), \ud s)}\\
   = & q(\phi(t,x), s,B).
  \end{align*}

  This completes the proof of the theorem.
\end{proof}

\subsection{Properties of the characteristic triple}

Let $X=\{X_t\}_{0\leqslant t<\tau}$ is a general PDMP with characteristic triple $(\phi,F,q)$. It follows from the subsection above that the property (i), (ii) and (iii) of $\phi$, $F$ and $q$ respectively can serve as the starting point for the development of the theory of general PDMPs. We firstly focus on the mapping $\phi$ which describes the deterministic evolution between two adjacent random jumps. Theorem \ref{thm.PDP=>str.Markov} (i) indicates that $\phi$ is in fact a semi-dynamic system (SDS) defined as follow.

\begin{definition}
  Let $(E,\E)$ be a Polish space, $\phi:(\R_+\times E,\B(\R_+)\times\E)\mapsto(E,\E)$ a measurable mapping. $\phi$ is called a \emph{semi-dynamic system} (or a \emph{semi-flow}) if for any $x\in E$, $s,t\in\R_+$ and $s+t<c(x)$,
  \begin{equation}\label{eq.SDS}
    \phi(0,x)=x,\quad \phi(t,\phi(s,x))=\phi(s+t,x),
  \end{equation}
  where
  \begin{equation}\label{eq.killing time}
    c(x)=c(\phi(t,x))+t,\quad 0\leqslant t<c(x).
  \end{equation}
\end{definition}

For convenience, we extend the state space by adding an isolated point $\Delta$ to $E$ and a part of the boundary of $E$. Let
\begin{equation*}
  \phi(c(x),x):=\left\{
    \begin{array}{ll}
    \displaystyle\lim_{t\uparrow c(x)}\phi(t,x), &
    \hbox{if $\displaystyle\lim_{t\uparrow c(x)}\phi(t,x)$ exists;} \\
    \Delta, & \hbox{otherwise.}
  \end{array}\right.
\end{equation*}
Thus, (\ref{eq.SDS}) still holds for $s+t=c(x)$ if $c(x)<\infty$.
Denote
$$\partial^+E:=\big\{\phi(c(x),x):x\in E\big\}.$$
We call $\partial^+E$ the \emph{flowing-out boundary} of the state space $E$, and let $\bar{E}:=E\cup\partial^+E$.

The subset of $\bar{E}$,
$\Phi_x:=\big\{\phi(t,x):t\in\mathcal{I}_x\big\}$ for any $x\in E$,
is called a \emph{path} (or a \emph{trajectory}) of the SDS $\phi$ starting from $x$ and killed at time $c(x)$.

A state $x\in E$ is called an \emph{equilibrium state} if $\phi(t,x)\equiv x$ holds for all $t\in \R_+$. An equilibrium state is also called a \textit{rest state} or a \emph{fixed state} somewhere. In this case, the path $\Phi_x$ contains only one state, that is, $\Phi_x=\{x\}$. The set of all the equilibrium states of the SDS $\phi$ is denoted by $E_e$.

A state $x\in E$ is called a \emph{periodic state} if $\phi(\cdot,x)$ is periodic. Denote minimal period by $T_x\in\R_+$. In this case, the path $\Phi_x$ is called to be \emph{periodic}. The set of all periodic states of the SDS $\phi$ is denoted by $E_p$. Note that an equilibrium state is a special case for $T_x=0$, thus $E_e\subset E_p$. From (\ref{eq.killing time}), it is easily to check that $c(x)=\infty$ for $x\in E_p$.

A state $x\in E\setminus E_p$ is called an \emph{aperiodic state}. For convenience, we let $T_x=\infty$ when $x$ is an aperiodic state.

The essential difference between an SDS and a dynamic system is that, starting from the initial state, an SDS can predict the future, but can not recall the history exactly. That is, an SDS allows the situation that different paths starting from different states join together at some states.

\begin{definition}\label{def.confluent state}
  A state $x\in E$ is called a \emph{confluent state} of an SDS $\phi$ if either one of the following conditions holds:
  \begin{enumerate}[(i)]
    \item there exist $x_1,x_2\in E\setminus E_p$ with $x_1\notin \Phi_{x_2}$, $x_2\notin\Phi_{x_1}$ and $\Phi_{x_1}\cap\Phi_{x_2}\neq \emptyset$ such that $x=\phi(t^*(x_1,x_2),x_1)$;
    \item there exists $x_1\in E\setminus E_p$ and $x_2\in E_p$ with $\Phi_{x_1}\supset\Phi_{x_2}$ such that $x=\phi(t^*(x_1,x_2),x_1)$,
  \end{enumerate}
  where $t^*(x_1,x_2):=\inf\{t>0:\phi(t,x_1)\in\Phi_{x_2}\}$.
\end{definition}

To consider the conditional survival function $F:E\times\R_+\mapsto[0,1]$ satisfying Theorem \ref{thm.PDP=>str.Markov} (ii) together with $F(x,0)=1$, let us recall the Stieltjes versions of exponentials and logarithms (see Meyer (1966) or Sharpe (1988), A.4).
In fact, $\{F(x,\cdot)\}_{x\in E}$ is a family of \emph{M-functions}, i.e., for any $x\in E$, $F(x,\cdot):[0,\infty)\mapsto[0,1]$ is a right continuous and decreasing function with $F(x,0)=1$. The \emph{Stieltjes logarithm} of the M-function $F(x,\cdot)$ is defined as
\begin{equation}\label{eq.slogF}
  \Lambda(x,t)=\slog F(x,t):=\int_{(0,t]}\!\frac{F(x,\ud s)}{F(x,s-)},\quad t\in\mathcal{I}_x.
\end{equation}
 Then $\{\Lambda(x,\cdot)\}_{x\in E}$ be a family of \emph{A-functions}, i.e., for any $x\in E$, $\Lambda(x,\cdot):[0,\infty)\mapsto[0,\infty]$ is a right continuous and increasing function such that $\Lambda(x,0)=0$, $\Delta\Lambda(x,t)<1$ for all $t$ with $\Lambda(x,t)<\Lambda(x,\infty)$, with $\Delta\Lambda(x,t)=1$ possible if $\Lambda(x,t)=\Lambda(x,\infty)<\infty$ (here $\Lambda(x,\infty):=\lim_{t\uparrow c(x)}\Lambda(x,t)$).
The \emph{Stieltjes exponential} of the A-function $\Lambda(x,\cdot)$ is defined as
\begin{equation}\label{eq.sexpLambda}
  \sexp \Lambda(x,t):=e^{-\Lambda^c(x,t)}\prod_{0<s\leqslant t}\big[1-\Delta\Lambda(x,s)\big],\quad t\in\mathcal{I}_x,
\end{equation}
where $\Lambda^c(x,\cdot)$ is the continuous part of $\Lambda(x,\cdot)$ and
$$\Delta\Lambda(x,t):=\Lambda(x,t)-\Lambda(x,t-),\quad t\in\mathcal{I}_x\setminus\{0\},\,x\in E.$$
It is known that the mapping $F\mapsto\slog F$ is a bijective mapping of the class of M-functions onto the class of A-functions, and that $\Lambda=\slog F$ if and only if $F=\sexp\Lambda$.

The function $\Lambda$ defined by the (\ref{eq.slogF})
is called the \emph{conditional hazard function}. Since the conditional survival function $F$ satisfies Theorem \ref{thm.PDP=>str.Markov} (ii), we have the following theorem.

\begin{theorem}\label{thm.A add<=>F mult}
  The conditional hazard function $\Lambda$ satisfies that for any $x\in E$, $s,t\in\R_+$ and $s+t\in\mathcal{I}_x$,
  \begin{equation}
    \Lambda(x,0)=0,\quad \Lambda(x,s+t)=\Lambda(x,s)+\Lambda(\phi(s,x),t)
  \end{equation}
  if and only if the conditional survival function $F$ satisfies that for any $x\in E$, $s,t\in\R_+$ and $s+t\in\mathcal{I}_x$,
  \begin{equation}
    F(x,0)=1,\quad F(x,s+t)=F(x,s)\,F(\phi(s,x),t).
  \end{equation}
\end{theorem}

\begin{proof}
  It follows directly from the definition of Stieltjes logarithm (\ref{eq.slogF}) and Stieltjes exponential (\ref{eq.sexpLambda}).
\end{proof}

\begin{remark}
  Since the characteristic $F$ is uniquely determined by $\Lambda$, the triple $(\phi, \Lambda, q)$ plays the same role. We call $(\phi, \Lambda,q)$ the characteristic triple too.
\end{remark}

Now we consider the transition kernel $q:E\times\R_+\times\E\mapsto[0,1]$ which satisfies Theorem \ref{thm.PDP=>str.Markov} (iii).

\begin{theorem}\label{thm.q=Q}
  Let the transition kernel $q:E\times\R_+\times\E\mapsto[0,1]$ satisfy Theorem \textnormal{\ref{thm.PDP=>str.Markov} (iii)}.
  There exists a measurable function $Q:\bar{E}\times \E\mapsto[0,1]$ such that, for any $x\in E$, $t\in\mathcal{I}_x\setminus\{0\}$, if $\phi(t,x)$ is not a confluent state, then
  \begin{equation}\label{eq.q=Q}
    q(x,t,B)=Q(\phi(t,x),B).
  \end{equation}
  Conversely, for the transition kernel $q$, if there exists a measurable function $Q$ such that \textnormal{(\ref{eq.q=Q})} holds for all $t\in\mathcal{I}_x\setminus\{0\}$ and $x\in E$, then $q$ satisfies Theorem \textnormal{\ref{thm.PDP=>str.Markov} (iii)}.
\end{theorem}

\begin{proof}
  Suppose $x^*\in E\setminus E_e$ is not a confluent state.
  Then, for any $x_1,x_2\in E$, $t_1\in\mathcal{I}_{x_1}\setminus\{0\}$, $t_2\in\mathcal{I}_{x_2}\setminus\{0\}$ satisfying $x^*=\phi(t_1,x_1)=\phi(t_2,x_2)$, there exists a $t_0$ with $0<t_0<t_1\land t_2$ such that
  $$x_0:=\phi(t_1-t_0,x_1)=\phi(t_2-t_0,x_2).$$
  Thus
  \begin{align*}
    q(x_1,t_1,B) & =q(\phi(t_1-t_0,x_1),t_0,B)=q(x_0,t_0,B) \\
   & =q(\phi(t_2-t_0,x_2),t_0,B)=q(x_2,t_2,B),\quad B\in\E.
  \end{align*}
  By the arbitrariness of $(x_1,t_1)$ and $(x_2,t_2)$, $q(x,t,B)$ is independent of the choice of $(x,t)$ with $\phi(t,x)=x^*$. Then there exists a measurable function $Q_1:\bar{E}\setminus E_e\times\E\mapsto[0,1]$ such that if $\phi(t,x)=x^*$ then
  \begin{equation*}\label{eq.q=Q1}
    q(x,t,B)=Q_1(x^*,B),\quad B\in\E.
  \end{equation*}

  Suppose that $x^*\in E_e$ is not a confluent state. That is, $\phi(t,x^*)=x^*$ holds for all $t\in\R_+$. Thus, for any $s,t\in\R_+\setminus\{0\}$,
  $$q(x^*,s+t,B)=q(\phi(s,x^*),t,B)=q(x^*,t,B),\quad B\in\E,$$
  which is independent of the choice of $s$ and $t$. Then, there exists a measurable function $Q_2:E_e\times\E\mapsto[0,1]$ such that
  \begin{equation*}\label{eq.q=Q2}
    q(x^*,t,B)=Q_2(x^*,B),\quad t\in\R_+,\,B\in\E.
  \end{equation*}

  Let
  $$Q(x,B)=\left\{\begin{array}{ll}
    Q_1(x,B), \quad & x\in \bar{E}\setminus E_e, B\in\E;\\
    Q_2(x,B), \quad & x\in  E_e, B\in\E.
  \end{array}\right.$$
  Then, for all $x\in E$, $t\in\mathcal{I}_x\setminus\{0\}$ and $B\in\E$, we have $q(x,t,B)=Q(\phi(t,x),B)$.

  Conversely, if (\ref{eq.q=Q}) holds for all $t\in\mathcal{I}_x\setminus\{0\}$ and $x\in E$, we have
  $$q(x,s+t,B)=Q(\phi(s+t,x),B)$$ and $$q(\phi(s,x),t,B)=Q(\phi(t,\phi(s,x)),B)=Q(\phi(s+t,x),B).$$
  Thus, $q(x,s+t,B)=q(\phi(s,x),t,B)$ for $s+t\in\mathcal{I}_x$, $x\in E$ and $B\in\E$.
  This completes the proof.
\end{proof}



\subsection{Sub-classes of general PDMPs and the characteristic $\Lambda$ }

The two well-studied sub-classes of general PDMPs, in the terminology of Jacod and Skorokhod \cite{jacod1996jumping}, are quasi-Hunt PDMPs and qusi-It\^o PDMPs defined as follows.

\begin{definition}
  A general PDMP $X$ is called a \emph{quasi-Hunt} if there is a jumping sequence $\{\tau_n\}$ with (\ref{eq.PiecewiseDeterministic}) such that the jumping time $\tau_n$ is totally inaccessible for all $n\geqslant 1$.
\end{definition}

\begin{theorem}
  Let $X$ be a general PDMP with the characteristic $\Lambda$. If for all $x\in E$ $\Lambda(x,\cdot)$ is continuous on $\mathcal{I}_x$, then the $X$ is quasi-Hunt. Conversely, if $\tau_1$ is totally inaccessible then $\Lambda(x,\cdot)$ is continuous on $\mathcal{I}_x$ for all $x\in E$.
\end{theorem}

\begin{proof}
  It follows directly from Jacod and Skorokhod \cite{jacod1996jumping} Theorem 5 (b) and Theorem 6 (c).
\end{proof}


%

\begin{definition}
  The quasi-Hunt PDMP $X$ is called \emph{quasi-It\^o} if for each $x\in E$ the predictable brackets of the locally square-integrable $\Prb_x$-martingales prior to $\tau$ are absolutely continuous w.r.t. Lebesgue measure.
\end{definition}

Jacod and Skorokhod \cite{jacod1996jumping} presents the necessary and sufficient condition for a quasi-Hunt PDMP to be quasi-It\^o in term of $\Lambda$ as follow.

\begin{theorem}
  The quasi-Hunt PDMP $X$ is quasi-It\^o if and only if $\Lambda(x,\cdot)$ is absolutely continuous on $\mathcal{I}_x$ for all $x\in E$.
\end{theorem}
\begin{proof}
  See Jacod and Skorokhod \cite{jacod1996jumping} Theorem 17.
\end{proof}

As supplementary, we introduce a new sub-class of PDMPs, the quasi-step PDMPs.

\begin{definition}
  A general PDMP $X$ is called \emph{quasi-step} if the jumping time $\tau_n$ is accessible for $n\geqslant 1$.
\end{definition}

\begin{theorem}
  Let $X$ be a general PDMP with the characteristic $\Lambda$. The $X$ is quasi-step if and only if $\Lambda(x,\cdot)$ is a step function for all $x\in E$, i.e.,
  \begin{equation}
    \Lambda(x,t)=\sum_{0<s\leqslant t}\Delta\Lambda(x,s),\quad t\in\mathcal{I}_x,\,x\in E,
  \end{equation}
  or equivalently
  \begin{equation}
    F(x,t)=\prod_{0<s\leqslant t}[1-\Delta\Lambda(x,s)],\quad t\in\mathcal{I}_x,\,x\in E.
  \end{equation}
\end{theorem}

\begin{proof}
  Let $X$ be a quasi-step PDMP. Then the jumping time $\tau_1$ is accessible. It follows from the definition of an accessible time that there exists a sequence of predictable times $\{T_n\}$ such that $\tau_1\in \{T_1,T_2,\cdots\}$ $\Prb_x$-a.s. for all $x\in E$. Since $\{T_n\}$ is a sequence of predictable times, there exists a sequence of $\F_0=\sigma(X_0)$-measurable r.v.'s, $\{R_n\}$ such that
  $$T_n\bbbone_{\{T_n\leqslant \tau_1\}}=R_n\bbbone_{\{R_n\leqslant \tau_1\}}.$$
  The $\sigma(X_0)$-measurability of $\{R_n\}$ implies that there exists a sequence of real measurable functions, $r_n: (E,\E)\mapsto(\R_+,\B(\R_+))$, such that $R_n=r_n(X_0)$, 
  or equivalently, $R_n=r_n(x)$. 
  We have
  $$T_n\bbbone_{\{T_n\leqslant \tau_1\}}=r_n(x)\bbbone_{\{r_n(x)\leqslant \tau_1\}}.$$
  These imply $\tau_1\in \{r_1(x),r_2(x),\cdots\}$ $\Prb_x$-a.s. for all $x\in E$. Hence,
  $$\Lambda(x,t)=\sum_{0<r_n(x)\leqslant t}\Delta\Lambda(x,r_n(x))$$
  for all $t\geqslant 0$ with
  $$\Delta\Lambda(x,r_n(x))=\frac{\Prb_x\{\tau_1=r_n(x)\}}{1-\sum_{r_m(x)<r_n(x)}\Prb_x\{\tau_1=r_m(x)\}}.$$

  Conversely, let $\Lambda(x,t)$ is a step function w.r.t. $t$ for all $x\in E$. Denote by $\{r_1(x), r_2(x),\cdots\}$ the set of all jumping points of $\Lambda(x,\cdot)$. And let $T_n=r_n(X_0)$. Then each $T_n$ is predictable and $\tau_1\in \{T_1,T_2,\cdots\}$ $\Prb_x$-a.s. That is to say, $\tau_1$ is accessible. Let $T_{mn}=\tau_{m-1}+r_n(X_{\tau_{m-1}})$ for $m,n\geqslant 1$.
  Obviously, each $T_{mn}$ is predictable and
  \begin{align*}
    \Prb_x\{\tau_m=T_{mn}\}&=\Prb_x\{\tau_m=\tau_{m-1}+r_n(X_{\tau_{m-1}})\}\\
  &=\Prb_x\{\tau_m-\tau_{m-1}=r_n(X_{\tau_{m-1}})\}\\
  &=\Ep_x\big[\Prb_x\{\tau_m-\tau_{m-1}=r_n(X_{\tau_{m-1}})|\F_{\tau_{m-1}}\}\big]\\
  &=\Ep_x\big[\Prb_{X_{\tau_{m-1}}}\{\tau_1=r_n(X_{\tau_{m-1}})\}\big].
  \end{align*}
  Taking summation on both ends in the equation above, we have
  $$\sum_{n=1}^{\infty}\Prb_x\{\tau_m=T_{mn}\}
  =\Ep_x\left[\sum_{n=1}^{\infty}\Prb_{X_{\tau_{m-1}}}\{\tau_1=r_n(X_{\tau_{m-1}})\}\right]=1.$$
  These imply that $\tau_m\in \{T_{m1}, T_{m2}, \cdots\}$ $\Prb_x$-a.s. for all $x\in E$. Equivalently, each $\tau_m$ is accessible. Hence, $X$ is quasi-step.
\end{proof}

\section{Additive functionals of an SDS}

Now we introduce the definition of additive functionals of an SDS which plays an important role in the study of PDMPs.

\begin{definition}\label{def.a.SDS}
  (i) For a given SDS $\phi$, a measurable function $a:E\times\R_+\mapsto\R$ is called an \emph{additive functional} of the SDS $\phi$ if the function $a(x,\cdot)$ is right-continuous on $\mathcal{I}_x$ for all $x\in E$ and
  \begin{equation}\label{eq.a.SDS}
    a(x,0)=0,\quad a(x,s)+a(\phi(s,x),t)=a(x,s+t),
  \end{equation}
  for any $x\in E$, $s,t\in\R_+$ and $s+t\in\mathcal{I}_x$. The set of all the additive functionals of the SDS $\phi$ is denoted by $\mathfrak{A}_\phi$. \\
  (ii) And additive functional $a$ of the SDS $\phi$ is called to be of \emph{locally finite variation} if $a(x,\cdot)$ is of finite variation on any compact subinterval of $\mathcal{I}_x$ for all $x\in E$. And denote by $\mathfrak{A}_\phi^{loc}$, the set of all the additive functionals of the SDS $\phi$ with locally finite variation.
\end{definition}

\begin{example}\label{ex.integrable&summable}
  (i) For a measurable function $g:\bar{E}\mapsto\R$,
    $$a(x,t):=\int_0^tg(\phi(s,x))\ud s,\quad t\in\mathcal{I}_x,\,x\in E$$
    is an additive functional of the SDS $\phi$ if it is well defined. In this case, $a\in\mathfrak{A}_\phi^{loc}$ if and only if the function $g$ is \emph{locally absolutely path-integrable}, i.e., 
    $$\int_0^t \big|g(\phi(s,x))\big|\ud s<\infty\quad \hbox{for all }t\in\mathcal{I}_x,\,x\in E.$$
  (ii) For a measurable function $g:\bar{E}\mapsto\R$,
    $$a(x,t):=\sum_{0<s\leqslant t}g(\phi(s,x)),\quad t\in\mathcal{I}_x,\,x\in E$$
    is an additive functional of the SDS $\phi$ if it is well defined. In this case, $a\in\mathfrak{A}_\phi^{loc}$ if and only if the function $g$ is \emph{locally absolutely path-summable}, i.e., 
    $$\sum_{0<s\leqslant t}\big|g(\phi(s,x))\big|<\infty\quad \hbox{for all }t\in\mathcal{I}_x,\,x\in E.$$
  (iii) For a measurable function $g:\bar{E}\mapsto\R$ and $a\in\mathfrak{A}_\phi^{loc}$,
    $$a^*(x,t):=\int_{(0,t]}g(\phi(s,x))a(x,\ud s),\quad t\in\mathcal{I}_x,\,x\in E$$
    is also an additive functional of the SDS $\phi$ if the integral above is well defined. In this case, $a^*\in\mathfrak{A}_\phi^{loc}$ if and only if 
    $$\int_{(0,t]}\big|g(\phi(s,x))\big|\,\big|a\big|(x,\ud s)<\infty\quad \hbox{for all }t\in\mathcal{I}_x,\,x\in E.$$
\end{example}

\begin{property}
  $\mathfrak{A}_\phi^{loc}$ is a linear space.
\end{property}

\begin{proof}
  It can be directly followed from the definition of an additive functional of the SDS $\phi$.
\end{proof}

\begin{property}
  Let $a\in\mathfrak{A}_\phi^{loc}$. Then there exists a function $g:E_e\mapsto\R$ such that for any equilibrium state $x\in E_e$,
  \begin{equation}\label{eq.a.equilbrium}
    a(x,t)=g(x)t,\quad t\in\R_+.
  \end{equation}
\end{property}

\begin{proof}
  Note that $c(x)=\infty$ in this case. The equation
  $$a(x,s)+a(\phi(s,x),t)=a(x,s+t)\quad \hbox{for }s,t\in\R_+$$
  becomes
  $$a(x,s)+a(x,t)=a(x,s+t)\quad \hbox{for } s,t\in\R_+.$$
  Thus, by the right-continuity, the conclusion is proved.
\end{proof}

\begin{property}
   Let $a\in\mathfrak{A}_\phi$. If $x\in E_p\setminus E_e$, that is, $x$ is a periodic state with minimal period $T_x\in\R_+\setminus\{0\}$, then
  \begin{equation}\label{eq.a.periodic}
    a(x,t)=\left\lfloor\frac{t}{T_x}\right\rfloor a(x,T_x)+a(x,t-\left\lfloor\frac{t}{T_x}\right\rfloor T_x),
  \end{equation}
  where $\left\lfloor t\right\rfloor$ is the largest integer not exceeding $t$.
\end{property}

\begin{proof}
  Set $n:=\left\lfloor\frac{t}{T_x}\right\rfloor$, $s:=t-nT_x$. Thus
  \begin{align*}
  a(x,t)&=a(x,nT_x+s)\\
        &=\sum_{k=1}^na(\phi((k-1)T_x,x),T_x)+a(\phi(nT_x,x),s)\\
        &=n\,a(x,T_x)+a(x,s).
  \end{align*}
\end{proof}

\begin{lemma}\label{lem.a.r-derivative}
  Let $a\in\mathfrak{A}_\phi$. 
  For any $x\in E$, $s\in[0,c(x))$, if the right derivative of function $a(x,t)$ at $t=s$ exists, then the right derivative of function $a(\phi(s,x),t)$ at $t=0$ exists, and
  \begin{equation}\label{eq.a.r-derivative}
  \frac{\partial^+a(x,t)}{\partial t}\bigg|_{t=s} =
  \frac{\partial^+a(\phi(s,x),t)}{\partial t}\bigg|_{t=0}.
  \end{equation}
\end{lemma}

\begin{proof}
  By the definition of additive functionals of an SDS,
  \begin{align*}
     \frac{\partial^+ a(x,t)}{\partial t}\bigg|_{t=s}
  &= \lim_{h\downarrow 0} \frac{a(x,s+h)-a(x,s)}{h}\\
  &= \lim_{h\downarrow 0} \frac{\big[a(x,s)+a(\phi(s,x),h)\big]-\big[a(x,s)+a(\phi(s,x),0)\big]}{h}\\
  &= \lim_{h\downarrow 0} \frac{a(\phi(s,x),h)-a(\phi(s,x),0)}{h}\\
  &= \frac{\partial^+ a(\phi(s,x),t)}{\partial t}\bigg|_{t=0}.
  \end{align*}
\end{proof}

For any $x\in E$, denote
\begin{equation}\label{eq.Xa}
  \mathcal{X}a(x):=
  \begin{cases}
    \displaystyle{\frac{\partial^+a(x,t)}{\partial t}\bigg|_{t=0}}, &
    \textrm{if}\; \displaystyle{\frac{\partial^+a(x,t)}{\partial t}\bigg|_{t=0}} \textrm{exists};\\
    0, & \textrm{otherwise}.\\
  \end{cases}
\end{equation}

\begin{theorem}
  Let $a\in\mathfrak{A}_\phi$. 
  $a(x,\cdot)$ is absolutely continuous on $\mathcal{I}_x$ for a fixed $x\in E$ if and only if
  \begin{equation}\label{eq.add.ac}
    a(x,t)=\int_0^t \mathcal{X}a(\phi(s,x)) \ud s,\quad t\in\mathcal{I}_x.
  \end{equation}
\end{theorem}

\begin{proof}
  The sufficiency is obvious. Now we will prove the necessity.
  For $x\in E$, if $a(x,\cdot)$ be a absolutely continuous function on $\mathcal{I}_x$, then it is differential on $\mathcal{I}_x$ almost everywhere. By Lemma \ref{lem.a.r-derivative}, we have
  $$\frac{\partial a(x,t)}{\partial t}=\frac{\partial^+a(x,t)}{\partial t}
  =\mathcal{X}a(\phi(t,x)) \quad \textrm{a.e. on }\mathcal{I}_x.$$
  Define a function
  $$\tilde{a}(x,t)=\int_0^t\mathcal{X}a(\phi(s,x)) \ud s,\quad t\in\mathcal{I}_x.$$
  Apparently, $\tilde{a}(x,\cdot)$ is absolutely continuous on $\mathcal{I}_x$, and
  $$\frac{\partial\tilde{a}(x,t)}{\partial t}=\mathcal{X}a(\phi(t,x))
  \quad\textrm{a.e. on }\mathcal{I}_x.$$
  Thus,
  $$\frac{\partial\tilde{a}(x,t)}{\partial t}=\frac{\partial a(x,t)}{\partial t}\quad\textrm{a.e. on }\mathcal{I}_x,$$
  or
  $$\frac{\partial}{\partial t}(\tilde{a}(x,t)-a(x,t))=0
  \quad\textrm{a.e. on }\mathcal{I}_x.$$
  Notice that $\tilde{a}(x,0)=a(x,0)=0$, by the absolute continuity of $\tilde{a}(x,\cdot)$ and $a(x,\cdot)$, we know that $\tilde{a}(x,\cdot)-a(x,\cdot)$ is absolutely continuous on $\mathcal{I}_x$, and that
  $$\tilde{a}(x,t)-a(x,t)=\tilde{a}(x,0)-a(x,0)=0\quad\hbox{for all }t\in\mathcal{I}_x.$$
  Hence
  $$a(x,t)=\tilde{a}(x,t),\quad t\in\mathcal{I}_x.$$
  This completes the proof.
\end{proof}

Let
$$J_a:=\big\{\phi(t,x):a(x,t)-a(x,t-)\neq 0,\; x\in E,\, t\in \mathcal{I}_x\!\setminus\!\{0\}\big\}.$$
A state $y\in \bar{E}$ is called an $a$-\emph{jumping state} if $y\in J_a$.

\begin{lemma}\label{lem.a.discrete}
  Let $a\in\mathfrak{A}_\phi$. 
  There exists a function $\Delta a$ such that, for any $x\in E$, $t\in\mathcal{I}_x$, if $\phi(t,x)$ is not a confluent state, then
  \begin{equation}\label{eq.a.Delta}
    a(x,t)-a(x,t-)=\Delta a(\phi(t,x)).
  \end{equation}
Furthermore, if $a\in\mathfrak{A}_\phi^{loc}$, then $\Delta a$ is locally absolutely path-summable.
\end{lemma}

\begin{proof}
  For any $x\in E$, $t\in\mathcal{I}_x$, denote $y=\phi(t,x)$. Apparently, if $y\notin J_a$, then we always have $\Delta a(y)=a(x,t)-a(x,t-)=0$.

  On the other hand, if $y\in J_a$ is not a confluent state, for any two different states $x_1,x_2\in E$ with $\phi(t_1,x_1)=\phi(t_2,x_2)=y$, there exists $s\in(0,t_1\land t_2)$ and $z\in E$ satisfying
  $$\phi(t_1-s,x_1)=\phi(t_2-s,x_2)=z\quad \hbox{and}\quad \phi(s,z)=y.$$
  Thus,
  \begin{align*}
   &a(x_1,t_1)-a(x_1,t_1-)\\
  =&\big[a(x_1,t_1-s)+a(z,s)\big]-\big[a(x_1,t_1-s)+a(z,s-)\big]\\
  =&a(z,s)-a(z,s-).
  \end{align*}
  Similarly, we have
  $$a(x_2,t_2)-a(x_2,t_2-)=a(z,s)-a(z,s-).$$
  Hence,
  $$a(x_1,t_1)-a(x_1,t_1-)=a(x_2,t_2)-a(x_2,t_2-).$$
  By the arbitrariness of $(x_1,t_1),(x_2,t_2)$, we can denote that $\Delta a(y)=a(x,t)-a(x,t-)$, which is a function of $y$ and independent of the choice of $(x,t)$ with $\phi(t,x)=y$.

  Moreover, if $a\in\mathfrak{A}_\phi^{loc}$, then
  $$\sum_{0<s\leqslant t}\big|\Delta a(\phi(s,x))\big|\leqslant\int_{(0,t]}\big|a\big|(x,\ud s)<\infty$$
  for all $x\in E$, $t\in\mathcal{I}_x$. This completes the proof.
\end{proof}

Next we will show the Lebesgue decomposition of an additive functional of the SDS $\phi$.

\begin{theorem}\label{thm.a.Lebesgue}
  Let $a\in\mathfrak{A}_\phi^{loc}$. Assume that $J_a$ contains no confluent state. Then there exists a locally absolutely path-integrable function $\mathcal{X}a$ and a locally absolutely path-summable function $\Delta a$ such that $a(x,\cdot)\,(x\in E)$ has the Legesgue decomposition
  \begin{equation}\label{eq.a.Lebesgue}
    a(x,\cdot)=a^{ac}(x,\cdot)+a^{sc}(x,\cdot)+a^{pd}(x,\cdot),
  \end{equation}
  where
  \begin{equation*}
    a^{ac}(x,t)=\int_0^t\mathcal{X}a(\phi(s,x))\ud s,\quad
    a^{pd}(x,t)=\sum_{0<s\leqslant t}\Delta a(\phi(s,x)),\qquad t\in\mathcal{I}_x,
  \end{equation*}
  and $a^{sc}(x,\cdot)$ is the singularly continuous part of $a(x,\cdot)$. Furthermore,
  $a^{ac}$, $a^{sc}$ and $a^{pd}$
  $\in\mathfrak{A}_\phi^{loc}$.
\end{theorem}

\begin{proof}
  For any $x\in E$, let $\mathcal{X}a$ be a function defined as (\ref{eq.Xa}). Since the derivatives of $a^{sc}(x,\cdot)$ and $a^{pd}(x,\cdot)$ are equal to $0$ almost everywhere,
  $$\frac{\partial a^{ac}(x,t)}{\partial t}=\frac{\partial^+a(x,t)}{\partial t}
  =\mathcal{X}a(\phi(t,x))\quad\textrm{a.e. on }\mathcal{I}_x,$$
  that is,
  $$a^{ac}(x,t)=\int_0^t\mathcal{X}a(\phi(s,x))\ud s,\quad t\in\mathcal{I}_x.$$
  On the other hand, $a^{ac}(x,\cdot)$ and $a^{sc}(x,\cdot)$ are both continuous, so all of $a$-jumping states coincide with $a^{pd}$-jumping states. By Lemma \ref{lem.a.discrete}, we have
  $$a^{pd}(x,t)=\sum_{0<s\leqslant t}a(x,s)-a(x,s-)
          =\sum_{0<s\leqslant t}\Delta a(\phi(s,x)),
          \quad t\in\mathcal{I}_x.$$
  Moreover, according to Example \ref{ex.integrable&summable}, $a^{ac}$ and $a^{pd}$ are both additive functionals of the SDS $\phi$, hence $a^{sc}$ is also an additive functional of the SDS $\phi$. The additive functional $a$ is of locally finite variation, so are $a^{ac}$, $a^{sc}$ and $a^{pd}$. Thus, $\mathcal{X}a$ and $\Delta a$ are locally absolutely path-integrable and locally absolutely path-summable, respectively.
\end{proof}


Notice that, by Theorem \ref{thm.A add<=>F mult}, the conditional hazard function $\Lambda$ defined as (\ref{eq.slogF}) is an additive functional of the SDS $\phi$. Then we have the following corollary.

\begin{corollary}\label{cor.Lambda.Lebesgue}
  Let $\Lambda$ be the conditional hazard function defined by \textnormal{(\ref{eq.slogF})}. Then $\Lambda\in\mathfrak{A}_\phi^{loc}$.
  Moreover, if $J_\Lambda$ contains no confluent state, then for any $x\in E$, $t\in\mathcal{I}_x$,
  \begin{equation}\label{eq.Lambda.Lebesgue}
    \Lambda(x,t)=\Lambda^{ac}(x,t)+\Lambda^{sc}(x,t)+\Lambda^{pd}(x,t),
  \end{equation}
  where
  $$\Lambda^{ac}(x,t)=\int_0^t\lambda(\phi(s,x))\ud s,\quad
  \Lambda^{pd}(x,t)=\sum_{0<s\leqslant t}\Delta\Lambda(\phi(s,x)),$$
  and $\Lambda^{sc}(x,\cdot)$ is the singularly continuous part of $\Lambda(x,\cdot)$.
  Or equivalently,
  \begin{equation}\label{eq.F.Lebesgue}
    F(x,t)=F^{ac}(x,t)\,F^{sc}(x,t)\,F^{pd}(x,t),
  \end{equation}
  where
  $$F^{ac}(x,t)=\exp\left[-\int_0^t\lambda(\phi(s,x))\ud s\right],\quad
  F^{sc}(x,t)=\exp\big[-\Lambda^{sc}(x,t)\big],$$
  $$F^{pd}(x,t)=\prod_{0<s\leqslant t}\big[1-\Delta\Lambda(\phi(s,x))\big].$$
  Furthermore,
  \begin{align*}
  &\lambda(x)=
    \left\{\begin{array}{ll}
      \displaystyle \frac{\partial^+\Lambda(x,t)}{\partial t}\bigg|_{t=0}, &
      \hbox{if }\displaystyle\frac{\partial^+\Lambda(x,t)}{\partial t}\bigg|_{t=0}\hbox{ exists;} \\
      0, & \hbox{otherwise,}
    \end{array}\right. \\
  &\Delta\Lambda(\phi(t,x))=\Lambda(x,t)-\Lambda(x,t-),
  \end{align*}
  are locally path-integrable and locally path-summable, respectively.
\end{corollary}

\begin{proof}
  Following from Theorem \ref{thm.A add<=>F mult}, we get $\Lambda\in\mathfrak{A}_\phi$. In addition, for any fixed $x\in E$, $F(x,\cdot)$ is a decreasing function, so $\Lambda(x,\cdot)$ is an increasing function, hence $\Lambda\in\mathfrak{A}_\phi^{loc}$. Thus, by Theorem \ref{thm.a.Lebesgue}, we get (\ref{eq.Lambda.Lebesgue}). And $\lambda$ and $\Delta\Lambda$ are locally path-integrable and locally path-summable, respectively. Applying (\ref{eq.sexpLambda}) to $F(x,\cdot)=\sexp\Lambda(x,\cdot)$, and noticing that $\Lambda^{ac}(x,\cdot)$ and $\Lambda^{sc}(x,\cdot)$ are continuous, we have (\ref{eq.F.Lebesgue}) directly.
\end{proof}

\begin{remark}
  Actually, from the corollary above, we have
  $$F^{ac}=\sexp\Lambda^{ac},\quad F^{sc}=\sexp\Lambda^{sc},\quad F^{pd}=\sexp\Lambda^{pd}.$$
  Since $\Lambda^{ac}$, $\Lambda^{sc}$ and $\Lambda^{pd}$ are all additive functionals of the SDS $\phi$, by Theorem \textnormal{\ref{thm.A add<=>F mult}}, $F^{ac}$, $F^{sc}$ and $F^{pd}$ all satisfy the condition \textnormal{(ii)} in Theorem \textnormal{\ref{thm.PDP=>str.Markov}}.
\end{remark}

\section{Additive functionals of a general PDMP}

\begin{definition}\label{def.A.PDMP}
  Let $X=\{X_t\}_{0\leqslant t<\tau}$ be a general PDMP with natural filtration $\{\F_t\}$.
  A real valued c\`adl\`ag $\{\F_t\}$-adapted process $A=\{A_t\}_{0\leqslant t<\tau}$ is called an \emph{additive functional} of $X$ if, for any stopping time $T$,
  \begin{enumerate}[(i)]
    \item $A_0=0$;
    \item $A_{T+s}=A_T+ A_s\circ\theta_T $ a.s. on $\{T<\infty\}$,
  \end{enumerate}
  where $\theta_t$ is the shift operator.
\end{definition}

\begin{theorem}\label{thm.A=a+b}
  Let $X$ be a general PDMP.
  An $\R$-valued c\`adl\`ag $\F$-adapted process $A=\{A_t\}_{0\leqslant t<\tau}$ with $A_0=0$ is an additive functional of $X$, if and only if there exists $a\in\mathfrak{A}_\phi$ and a measurable function $b:E\times\R_+\times E\mapsto\R$ satisfying that for any $x,y\in E$, $s,t\in\R_+$ and $s+t\in\mathcal{I}_x\setminus\{0\}$,
  \begin{equation}\label{eq.b.invariant}
    b(x,s+t,y)=b(\phi(s,x),t,y),
  \end{equation}
  such that
  \begin{equation}\label{eq.A=a+b}
    A_t=A_{\tau_n} + a(X_{\tau_n},t-\tau_n)
    +b(X_{\tau_n},\tau_{n+1}-\tau_n,X_{\tau_{n+1}})\bbbone_{\{t=\tau_{n+1}\}}
  \end{equation}
  for $t\in(\tau_n,\tau_{n+1}]$, $n\in\N$. Especially, for any $x\in E$, $t\in\mathcal{I}_x\setminus\{0\}$, if $\phi(t,x)$ is not a confluent state, then there exist a measurable function $\bar{b}:\bar{E}\times E\mapsto \R$ such that
  \begin{equation}\label{eq.b=bar.b}
    b(x,t,y)=\bar{b}(\phi(t,x),y),\quad y\in E.
  \end{equation}
\end{theorem}

\begin{proof}
  The sufficiency is obvious. We only need to prove the necessity. By the additivity of $A$, It suffices to prove (\ref{eq.A=a+b}) for n=0.
  Since $\F_t=\sigma(X_0)$ for all $t<\tau_1$, and $A$ is $\F$-adapted, there exists a measurable function $a:\bar{E}\times\R_+\mapsto \R$ such that $A_t=a(X_0,t)$ for all $t<\tau_1$. It is followed from Definition \ref{def.A.PDMP}, that $a$ is an additive functional of the SDS $\phi$. 

  Define a process $A^-=\{A_t^-\}_{0\leqslant t<\tau}$ by
  \begin{align*}
    A_0^- &:= 0, \\
    A_t^- &:= A_{\tau_n}+a(X_{\tau_n},t-\tau_n),\quad t\in(\tau_n,\tau_{n+1}],\;n\in\N.
  \end{align*}
  Apparently, $A^-$ is an additive functional of $X$.
  And denote
  $$\hat{A}_t:=A_t-A_t^-,\quad 0\leqslant t<\tau.$$
  Since $\F_{\tau_1}=\sigma(X_0,\tau_1,X_{\tau_1})$, there exists a measurable function $b:E\times\R_+\times E\mapsto \R$ such that
  $$\hat{A}_{\tau_1}=b(X_0,\tau_1,X_{\tau_1})\quad\hbox{ a.s. on }\{\tau_1<\infty\}.$$
  By the additivity of $A$ and $A^-$,
  $$ \hat{A}_{\tau_1}= \hat{A}_{\tau_1}\circ\theta_s \quad\hbox{ a.s. on }\{s<\tau_1\},$$
  then we get (\ref{eq.b.invariant}). Hence, (\ref{eq.A=a+b}) holds for $n=0$.

  Following from (\ref{eq.b.invariant}), to prove (\ref{eq.b=bar.b}), we need to consider whether $\phi(t,x)$ is an equilibrium state or not for $x\in E$ and $t\in\mathcal{I}_x\setminus\{0\}$. If $x'\in E_e$ is not a confluent state, that is, $\phi(t,x')\equiv x'$ for all $t\in\R_+$, then
  $$b(x',s+t,y)=b(\phi(s,x'),t,y)=b(x',t,y)$$
  holds for all $y\in E$, $s,t\in\R_+$ and $s+t>0$. Thus there exists a measurable function $\bar{b}_1:E_e\times E\mapsto\R$ such that
  $$b(x',t,y)=\bar{b}_1(x',y)$$
  which is independent of $t$. On the other hand, if $x''\in E\setminus E_e$ is not a confluent state. For any $x_1,x_2\in E$, $t_1\in\mathcal{I}_{x_1}\setminus\{0\}$ and $t_2\in\mathcal{I}_{x_2}\setminus\{0\}$ satisfying $\phi(t_1,x_1)=\phi(t_2,x_2)=x''$, there exists $t_0\in(0,t_1\land t_2)$ such that
  $$\phi(t_1-t_0,x_1)=\phi(t_2-t_0,x_2)=:x_0.$$
  Thus, for $y\in E$,
  \begin{align*}
    b(x_1,t_1,y)=&b(\phi(t_1-t_0,x_1),t_0,y)=b(x_0,t_0,y)\\
                =&b(\phi(t_2-t_0,x_2),t_0,y)=b(x_2,t_2,y).
  \end{align*}
  By the arbitrariness of $(x_1,t_1)$ and $(x_2,t_2)$, there exists a measurable function $\bar{b}_2:E\setminus E_e\times E\mapsto\R$ such that
  $$b(x,t,y)=\bar{b}_2(x'',y)$$
  holds for all $x\in E$, $t\in\mathcal{I}_x\setminus\{0\}$ with $\phi(t,x)=x''$.
  Therefore, for all $x\in E$, $t\in\mathcal{I}_x\setminus\{0\}$, if $\phi(t,x)$ is not a confluent state, then we have (\ref{eq.b=bar.b}). Actually, we can take
  $$\bar{b}(x,\cdot)=
  \left\{
    \begin{array}{ll}
      \bar{b}_1(x,\cdot), & x\in E_e; \\
      \bar{b}_2(x,\cdot), & x\in E\setminus E_e.
    \end{array}
  \right.$$
  The proof is complete.
\end{proof}

\begin{corollary}\label{cor.A.pred=a}
  Let $A$ be an additive functional of a general PDMP $X$. $A$ is predictable if and only if there exists $a\in\mathfrak{A}_\phi$ such that
  \begin{equation}\label{eq.A.pred=a}
    A_t=A_{\tau_n} + a(X_{\tau_n},t-\tau_n),\quad t\in(\tau_n,\tau_{n+1}],\, n\in\N.
  \end{equation}
\end{corollary}

\begin{proof}
  The sufficiency is obvious. Now we prove the necessity. By the additivity of $A$, it suffices to prove (\ref{eq.A.pred=a}) for n=0.
  Since $\F_t=\sigma(X_0)$ for all $t<\tau_1$, and $A$ is predictable, there exists a measurable function $a:\bar{E}\times\R_+\mapsto \R$ such that $A_t=a(X_0,t)$ for all $t\leqslant\tau_1$. The additivity of $A$ implies that $a$ is an additive functional of the SDS $\phi$.
\end{proof}

Here we define an auxiliary process $X^-=\{X_t^-\}_{0\leqslant t<\tau}$,
\begin{equation*}
  X_t^-:=X_0\bbbone_{\{t=0\}}+\sum_{n=0}^\infty\phi(t-\tau_n,X_{\tau_n})
  \bbbone_{\{\tau_n<t\leqslant\tau_{n+1}\}},  \quad 0\leqslant t<\tau.
\end{equation*}
Obviously, $X_t^-=X_t$ if $t\neq\tau_n$, $n=1,2,\dots$ In other words, $X^-$ modifies the values of $X$ only at the random jumping times. It is known that $X^-$ is predictable.

\begin{corollary}\label{thm.A.Lebesgue}
  Let $A$ be an optional additive functional of the general PDMP $X$ with the associated $a$ and $b$ as in Theorem \textnormal{\ref{thm.A=a+b}}. Assume that $J_a$ contains no confluent state. If $a\in\mathfrak{A}_\phi^{loc}$, then
  \begin{align}
    A_t=&\int_0^t\mathcal{X}a(X_s)\ud s+\sum_{n=0}^{N_t}a^{sc}(X_{\tau_n},t\land\tau_{n+1}-\tau_n)
    +\sum_{0<s\leqslant t}\Delta a(X_s^-) \nonumber\\
    &+\sum_{n=1}^{N_t}b(X_{\tau_{n-1}},\tau_{n}-\tau_{n-1},X_{\tau_{n}}),\quad\quad 0\leqslant t<\tau, \label{eq.A.Lebesgue}
  \end{align}
  where $N_t:=\max\{n\geqslant 0:\tau_n\leqslant t\}$, $a^{sc}\in\mathfrak{A}_\phi^{loc}$, $\mathcal{X}a$ is locally absolutely path-integrable, and $\Delta a$ is locally absolutely path-summable.
\end{corollary}

\begin{proof}
  For an optional additive functional $A$, we have (\ref{eq.A=a+b}). If $a\in\mathfrak{A}_\phi^{loc}$, applying Theorem \ref{thm.a.Lebesgue}, we get (\ref{eq.A.Lebesgue}).
\end{proof}

\begin{definition}
  The process $M=\{M_t\}_{0\leqslant t<\tau}$ is a \emph{local martingale prior to $\tau$} if there exist localizing times $\{T_n\}$ with $T_n\uparrow\tau$ such that $M_{t\land T_n}$ is a martingale for each $n\in\N$.
\end{definition}

Let $X$ be a general PDMP with characteristic triple $(\phi,F,q)$. The corresponding \emph{jumping measure} is
\begin{equation}\label{eq.mu}
  \mu(\ud t,\ud y)=\sum_{n=1}^\infty \delta_{(\tau_n,X_{\tau_n})}(\ud t,\ud y) \bbbone_{\{\tau_n<\infty\}}.
\end{equation}
The \emph{predictable dual projection} (or \emph{compensator}) of the jumping measure $\mu$ for a general PDMP $X$ is
$$\nu(\ud t,\ud y) = \sum_{n=0}^\infty\frac{F(X_{\tau_n},\ud t-\tau_n)\, q(X_{\tau_n},t-\tau_n,\ud y)}{F(X_{\tau_n},(t-\tau_n)-)} \bbbone_{\{\tau_n<t\leqslant\tau_{n+1}\}}.$$
Due to (\ref{eq.slogF}), the definition of Stieltjes logarithm, we have
$$\frac{F(X_{\tau_n},\ud t-\tau_n)}{F(X_{\tau_n},(t-\tau_n)-)}
  = \Lambda(X_{\tau_n},\ud t-\tau_n).$$
Therefore,
\begin{equation}\label{eq.nu}
  \nu(\ud t,\ud y)=\sum_{n=0}^\infty\Lambda(X_{\tau_n},\ud t-\tau_n)\,
  q(X_{\tau_n},t-\tau_n,\ud y)\bbbone_{\{\tau_n<t\leqslant\tau_{n+1}\}}.
\end{equation}

\begin{theorem}\label{thm.semimartingale}
  Let $A$ be an additive functional of a general PDMP $X$, and $a,b$ the functions associated with $A$ as Theorem \textnormal{\ref{thm.A=a+b}}.


  \textnormal{(i)} $A$ is a local martingale prior to $\tau$ if and only if for any $x\in E$ and $t\in\mathcal{I}_x$, we have
  \begin{equation}\label{eq.local-m 1}
    \int_{(0,t]}\int_E \big|b(x,s,y)\big| q(x,s,\ud y) \Lambda(x,\ud s)<\infty,
  \end{equation}
  and
  \begin{equation}\label{eq.local-m 2}
    a(x,t)=-\int_{(0,t]}\int_E b(x,s,y)\, q(x,s,\ud y)\, \Lambda(x,\ud s).
  \end{equation}

  \textnormal{(ii)} $A$ is a special semimartingale prior to $\tau$ if and only if $a\in\mathfrak{A}_\phi^{loc}$, and \textnormal{(\ref{eq.local-m 1})} is satisfied. In this case, $A$ has a canonical decomposition $A=M+B$, where $M$ is a local martingale prior to $\tau$ which is also an additive functional of $X$, and $B$ is a predictable additive functional of $X$ with $B_0=0$ and
  \begin{equation}\label{eq.predictable B}
    B_t=B_{\tau_n}+a^*(X_{\tau_n},t-\tau_n),\quad t\in(\tau_n,\tau_{n+1}],\,n\in\N,
  \end{equation}
  where
  \begin{equation*}
    a^*(x,t)=a(x,t)+\int_{(0,t]}\int_E b(x,s,y)\,q(x,s,\ud y)\,\Lambda(x,\ud s).
  \end{equation*}
\end{theorem}

\begin{proof}

  (ii) It follows from Theorem 14 a) of \cite{jacod1996jumping} that $A$ is a local martingale prior to $\tau$ if and only if there exists a predictable function $\bar{A}:\Omega\times\R_+\times E\mapsto\R$ such that
  \begin{equation}\label{eq.local-m.ineq}
    \int_{(0,t]\times E}\big|\bar{A}(s,y)\big|\,\big|\mu-\nu\big|(\ud s,\ud y)<\infty \quad \Prb_x\hbox{-a.s. for }0\leqslant t<\tau
  \end{equation}
  and
  \begin{align}
    A_t&=\int_{(0,t]\times E}\bar{A}(s,y)(\mu-\nu)(\ud s,\ud y)\nonumber\\
       &=\sum_{n=1}^{N_t}\bar{A}(\tau_n,X_{\tau_n})-\int_{(0,t]\times E}\bar{A}(s,y)\nu(\ud s,\ud y),\quad 0\leqslant t<\tau. \label{eq.local-m.eq}
  \end{align}

  If $A$ is a local martingale prior to $\tau$, comparing (\ref{eq.local-m.eq}) with (\ref{eq.A=a+b}), we have
  $$\bar{A}(\tau_{n+1},X_{\tau_{n+1}})=b(X_{\tau_n},\tau_{n+1}-\tau_n,X_{\tau_{n+1}}),$$
  and
  \begin{align*}
     &a(X_{\tau_n},t-\tau_n)\\
    =&-\int_{(\tau_n,t]\times E}\bar{A}(s,y)\nu(\ud s,\ud y)\\
    =&-\int_{(0,t-\tau_n]}\int_E b(X_{\tau_n},u,y)\,q(X_{\tau_n},u,\ud y)\,\Lambda(X_{\tau_n},\ud u), \quad\tau_n<t\leqslant\tau_{n+1},
  \end{align*}
  for all $n\in\N$. Thus, we get (\ref{eq.local-m 2}). The condition (\ref{eq.local-m.ineq}) is equivalent to that $A$ is of locally finite variation. And that $A$ is of locally finite variation is also equivalent to that $a$ is of locally finite variation, which is the same as (\ref{eq.local-m 1}).

  Conversely, if A is an additive functional with $a$ and $b$,  and (\ref{eq.local-m 1}) and (\ref{eq.local-m 2}) are satisfied, by taking
  $$\bar{A}(s,y)=\sum_{n=0}^\infty\bar{A}_n(s,y)\bbbone_{\{\tau_n<s\leqslant\tau_{n+1}\}},$$
  where
  $\bar{A}_n(s,y)=b(X_{\tau_n},s-\tau_n,y),$
  we have (\ref{eq.local-m.ineq}) and (\ref{eq.local-m.eq}). Then $A$ is a local martingale prior to $\tau$.

  (iii) Assume first that $A$ is a special semimartingale prior to $\tau$. Since all local martingales prior to $\tau$ are of locally finite variation, equivalently, all the semi-martingales prior to $\tau$ are of locally finite variation, which implies that $a$ has locally finite variation. And its canonical decomposition $A=M+B$ gives a local martingale $M$ prior to $\tau$ and a predictable process $B$, while both of them are additive functionals of $X$. Call $(a_M,b_M)$ and $(a_B,b_B)$ the terms associated with $M$ and $B$ in Theorem \ref{thm.A=a+b}. Clearly, $a=a_M+a_B$ and $b=b_M+b_B$.
  Since $B$ is predictable, we have $\Delta B_{\tau_n}=0$, hence
  $\Delta M_{\tau_n}=\Delta A_{\tau_n}$ for $n\geqslant 1$.
  Therefore, $b_B=0$ and $b_M=b$.
  For local martingale $M$ prior to $\tau$, applying (i), we obtain (\ref{eq.local-m 1}) and also that $a_M$ is given by the right-hand side of (\ref{eq.local-m 2}). Therefore (\ref{eq.predictable B}) follows from $a_B=a-a_M$ and $b_B=0$.

  For the converse, assume that we have (\ref{eq.local-m 1}) and  $a\in\mathfrak{A}_\phi^{loc}$. Set $b_M=b$ and define $a_M$ by the right-hand side of (\ref{eq.local-m 2}), then $(a_M,b_M)$ is associated with an additive local martingale $M$ prior to $\tau$ by (i). If $B$ is given by (\ref{eq.predictable B}), which is a predictable process associated with $a_B=a-a_M$ and $b_B=0$, and since $a_B$ is an additive functional of the SDS $\phi$, $B$ is a predictable additive functional of $X$. The assumptions ensure that $a_B$ is of locally finite variation. Equivalently, $B$ is also of locally finite variation. Then $A=M+B$ is a special semimartingale prior to $\tau$.
\end{proof}

Let the PDMP $X$ be quasi-Hunt. 
We prefer the characteristic triple ($\phi, \Lambda, Q$) to ($\phi, \Lambda, q$) in the quasi-Hunt cases in the following. The following theorem, which is Theorem 18 in \cite{jacod1996jumping}, is a direct corollary of Theorem \ref{thm.semimartingale}.

\begin{theorem}\label{cor.semimartingale.Hunt}
  Let $A$ be an additive functional of a quasi-Hunt PDMP $X$ with the characteristic triple $(\phi, \Lambda, Q)$, and $a,\bar{b}$ the functions associated with $A$ as Theorem \textnormal{\ref{thm.A=a+b}}.


  \textnormal{(i)} $A$ is a local martingale prior to $\tau$ if and only if for any $x\in E$ and $t\in[0,c(x))$, we have
  \begin{equation}\label{eq.local-m 1.Hunt}
    \int_{(0,t]}\int_E \big|\bar{b}(\phi(s,x),y)\big| Q(\phi(s,x),\ud y) \Lambda(x,\ud s)<\infty,
  \end{equation}
  and
  \begin{equation}\label{eq.local-m 2.Hunt}
    a(x,t)=-\int_{(0,t]}\int_E \bar{b}(\phi(s,x),y)\, Q(\phi(s,x),\ud y)\, \Lambda(x,\ud s).
  \end{equation}

  \textnormal{(ii)} $A$ is a special semimartingale prior to $\tau$ if and only if the associated $a$ has locally finite variation and \textnormal{(\ref{eq.local-m 1.Hunt})} is satisfied. In this case, $A$ has a canonical decomposition $A=M+B$, where $M$ is a local martingale prior to $\tau$ which is also an additive functional of $X$, and $B$ is a predictable additive functional of $X$ with $B_0=0$ and
  \begin{equation}\label{eq.predictable B.Hunt}
    B_t=B_{\tau_n}+a^*(X_{\tau_n},t-\tau_n),\quad t\in(\tau_n,\tau_{n+1}],\,n\in\N,
  \end{equation}
  where
  \begin{equation}
    a^*(x,t)=a(x,t)+\int_{(0,t]}\int_E \bar{b}(\phi(s,x),y)\,Q(\phi(s,x),\ud y)\,\Lambda(x,\ud s).
  \end{equation}
\end{theorem}

\begin{corollary}\label{cor.semimartingale.step}
  Let $A$ be an additive functional of a quasi-step PDMP $X$ with the characteristic triple $(\phi,\Lambda,Q)$, and $a,\bar{b}$ the functions associated with $A$ as Theorem \textnormal{\ref{thm.A=a+b}}. Assume that $J_\Lambda$ contains no confluent state.


  \textnormal{(i)} $A$ is a local martingale prior to $\tau$ if and only if for any $x\in E$ and $t\in\mathcal{I}_x$, we have
  \begin{equation}\label{eq.local-m 1.step}
    \sum_{0<s\leqslant t}\Delta\Lambda(\phi(s,x))\int_E \big|\bar{b}(\phi(s,x),y)\big| Q(\phi(s,x),\ud y)<\infty,
  \end{equation}
  and
  \begin{equation}\label{eq.local-m 2.step}
    a(x,t)=-\sum_{0<s\leqslant t}\Delta\Lambda(\phi(s,x))
    \int_E \bar{b}(\phi(s,x),y)\,Q(\phi(s,x),\ud y).
  \end{equation}

  \textnormal{(ii)} $A$ is a special semimartingale prior to $\tau$ if and only if the corresponding $a$ has locally finite variation and \textnormal{(\ref{eq.local-m 1.step})} is satisfied. In this case, $A$ has a canonical decomposition $A=M+B$, where $M$ is a local martingale prior to $\tau$ which is also an additive functional of $X$, and $B$ is a predictable additive functional of $X$ with $B_0=0$ and
  \begin{equation}\label{eq.predictable B.step}
    B_t=B_{\tau_n}+a^*(X_{\tau_n},t-\tau_n),\quad t\in(\tau_n,\tau_{n+1}],\,n\in\N,
  \end{equation}
  where
  \begin{equation}
    a^*(x,t)=a(x,t)+\sum_{0<s\leqslant t}\Delta\Lambda(\phi(s,x))\int_E \bar{b}(\phi(s,x),y)\,Q(\phi(s,x),\ud y).
  \end{equation}
\end{corollary}

\begin{proof}
  Following from Theorem \ref{thm.semimartingale}, and noticing that
  $$\Lambda(x,t)=\sum_{0<s\leqslant t}\Delta\Lambda(\phi(s,x)),\quad t\in\mathcal{I}_x,\,x\in E$$
  for the quasi-step PDMP $X$, we get the conclusions.
\end{proof}

\begin{definition}
  A local martingale $M=\{M_t\}_{0\leqslant t<\tau}$ prior to $\tau$ is called to be \emph{additive} if for any stopping time $T$,
  $$M_0=0,\quad M_{T+s}=M_T+M_s\circ\theta_T\quad \hbox{a.s. on }\{T<\infty\}.$$
\end{definition}

\begin{theorem}
  Let $X$ be a general PDMP. Assume that $J_\Lambda$ contains no confluent state. $X$ is quasi-step if and only if for each $x\in E$ all the additive $\Prb_x$-local martingales prior to $\tau$ are step processes.
\end{theorem}

\begin{proof}
  If $X$ is a quasi-step PDMP, we have $F=F^{pd}$, or equivalently, $\Lambda=\Lambda^{pd}$. Then, following from Theorem \ref{cor.semimartingale.step} (i), for each $x\in E$, a $\Prb_x$-local martingale $M$ can be represented by
  \begin{align*}
    M_t= & -\sum_{0<s\leqslant t}\Delta\Lambda(X_s^-)\int_E\bar{b}(X_s^-,y)\,Q(X_s^-,\ud y)+\sum_{n=1}^{N_t}\bar{b}(X_{\tau_n}^-,X_{\tau_n}),\quad t\in[0,\tau),
  \end{align*}
  which is a step process.

  Conversely, let $X$ be a general PDMP. For any $\Prb_x$-local martingale $M$ prior to $\tau$, by Corollary \ref{thm.semimartingale} (i), we have
  \begin{align*}
    M_t= & M_{\tau_n}-\int_{(\tau_n,t]}\int_Eb(X_{\tau_n},s-\tau_n,y)\,q(X_{\tau_n},s-\tau_n,\ud y)\,\Lambda(X_{\tau_n},\ud s-\tau_n) \\
    & +b(X_{\tau_n},\tau_{n+1}-\tau_n,X_{\tau_{n+1}})\,\bbbone_{\{t=\tau_{n+1}\}},
    \qquad\qquad t\in(\tau_n,\tau_{n+1}],
  \end{align*}
  for any $n\in\N$. If $M$ is a step process, then the second term must be a summation along $(\tau_n,t]$, that is, $\Lambda=\Lambda^{pd}$. This completes the proof.
\end{proof}

\section{Measure-valued generator for PDMPs}

\subsection{Measure-valued generater}

\begin{definition}
  A measurable function $f:\bar{E}\mapsto\R$ is called to be of \emph{locally path-finite variation} if $f(\phi(\cdot,x))$ is right-continuous and of finite variation on any compact subinterval of $\mathcal{I}_x$ for any fixed $x\in E$. And denote by $\mathfrak{V}_{\phi}^{loc}$, the set of all locally path-finite variation function $f$'s.
\end{definition}

For a given $f\in\mathfrak{V}_{\phi}^{loc}$, define an operator $D$ by
$$Df(x,t):=f(\phi(t,x))-f(x),\quad t\in\mathcal{I}_x,\,x\in E.$$
Apparently, $f\in\mathfrak{V}_{\phi}^{loc}$ if and only if $Df\in \mathfrak{A}_{\phi}^{loc}$.

\begin{lemma}\label{lem.f-f is A}
  Let $X=\{X_t\}_{0\leqslant t<\tau}$ be a general PDMP, and $f:\bar{E}\mapsto\R$ a measurable function. Then the process $\{f(X_t)-f(X_0)\}_{0\leqslant t<\tau}$ is an additive functional of $X$ with
  \begin{equation}\label{eq.A^f=Df+bf}
    f(X_t)=f(X_{\tau_n})+Df(X_{\tau_n},t-\tau_n)+[f(X_{\tau_{n+1}})-f(X_{\tau_{n+1}}^-)]\, \bbbone_{\{t=\tau_{n+1}\}}
  \end{equation}
  for $t\in(\tau_n,\tau_{n+1}]$, $n\in\N$. Furthermore, $\{f(X_t)-f(X_0)\}_{0\leqslant t<\tau}$ is a.s. of locally finite variation on $[0,\tau)$ if and only if $f\in\mathfrak{V}_{\phi}^{loc}$.
\end{lemma}

\begin{proof}
  As $\tau_n<t<\tau_{n+1}$,
  $$f(X_t)-f(X_{\tau_n})=f(\phi(t-\tau_n,X_{\tau_n}))-f(X_{\tau_n})=Df(X_{\tau_n},t-\tau_n);$$
  and, in view of $X_{\tau_{n+1}}^-=\phi(\tau_{n+1}-\tau_n,X_{\tau_n})$, we have
  \begin{align*}
    f(X_{\tau_{n+1}})-f(X_{\tau_n})&=f(X_{\tau_{n+1}})-f(X_{\tau_{n+1}}^-)+f(X_{\tau_{n+1}}^-)-f(X_{\tau_n})\\
  &=f(X_{\tau_{n+1}})-f(X_{\tau_{n+1}}^-)+f(\phi(\tau_{n+1}-\tau_n,X_{\tau_n})-f(X_{\tau_n})\\
  &=f(X_{\tau_{n+1}})-f(X_{\tau_{n+1}}^-)+Df(X_{\tau_n},\tau_{n+1}-\tau_n).
  \end{align*}
  Hence, (\ref{eq.A^f=Df+bf}) holds.

  Let $A_t^f:=f(X_t)-f(X_0)$ for $t\in[0,\tau)$. It is obvious that
  $$A_t^f=A_{\tau_n}^f+a_f(X_{\tau_n},t-\tau_n)+b_f(X_{\tau_n},\tau_{n+1}-\tau_n,X_{\tau_{n+1}}) \bbbone_{\{t=\tau_{n+1}\}}$$
  holds for $t\in(\tau_n,\tau_{n+1}]$, $n\in\N$.
  It follows from Theorem \ref{thm.A=a+b} with $a_f=Df\in \mathfrak{A}_{\phi}$ and $b_f(x,t,y)=f(y)-f(\phi(t,x))$ that the process $\{f(X_t)-f(X_0)\}_{0\leqslant t<\tau}$ is an additive functional of $X$.

  Moreover, the sufficient and necessary condition of $A^f$ being a.s. of locally finite variation on $[0,\tau)$ is that $Df\in\mathfrak{A}_\phi^{loc}$, which is equivalent to that $f\in\mathfrak{V}_\phi^{loc}$.
\end{proof}

\begin{definition}\label{def.measure-valued generator}
  Let $X=\{X_t\}_{0\leqslant t<\tau}$ be a general PDMP with characteristic triple $(\phi,F,q)$. $\D(\A)$ denotes the set of $f\in\mathfrak{V}_{\phi}^{loc}$ with the following property: there exists $a\in \mathfrak{A}_{\phi}^{loc}$ such that the process
  \begin{equation}\label{eq.generator.def}
    M_t^f:=f(X_t)-f(X_0)-\sum_{n=0}^{N_t} a(X_{\tau_n},t\land\tau_{n+1}-\tau_n), \quad t<\tau
  \end{equation}
  is a local martingale prior to $\tau$. 
  Then we denote that $a=\A f$, and call $(\A,\D(\A))$ the \emph{measure-valued generator} of the general PDMP $X$.
\end{definition}

For $a\in \mathfrak{A}_{\phi}^{loc}$, because $a(x,t)$ is of locally finite variation with respect to $t$, so $a(x,\cdot)$ can also be treated as a $\sigma$-finite signed measure on $\mathcal{I}_x$ for all $x\in E$. This is the reason why we call the operator $\A$ is measure-valued.

\begin{theorem}\label{thm.measure-valued generator}
  Let $X$ be a general PDMP. Then the domain $\D(\A)$ of the measure-valued generator $\A$ of $X$ consists of all $f\in\mathfrak{V}_{\phi}^{loc}$ such that for any $x\in E$, $t\in\mathcal{I}_x$,
  \begin{equation}\label{eq.domain.A-general}
    \int_{(0,t]}\int_E\big|f(y)-f(\phi(s,x))\big|q(x,s,\ud y)\Lambda(x,\ud s)<\infty.
  \end{equation}
  Moreover, if $f\in\D(\A)$, then for any $x\in E$ and $t\in\mathcal{I}_x$,
  \begin{equation}\label{eq.measure-valued generator}
    \A f(x,t)=Df(x,t)+\int_{(0,t]}\int_E\big[f(y)-f(\phi(t,x))\big]q(x,s,\ud y)\Lambda(x,\ud s).
  \end{equation}
\end{theorem}

\begin{proof}
  Suppose that $f\in\D(\A)$. Then
  $$M^f_t=f(X_t)-f(X_0)-\sum_{n=0}^{N_t}\A f(X_{\tau_n},t\land\tau_{n+1}-\tau_n),\quad t<\tau$$
  is a local martingale prior to $\tau$. Moreover, by Lemma \ref{lem.f-f is A},
  \begin{align*}
    M^f_t=&\sum_{n=0}^{N_t}Df(X_{\tau_n},t\land\tau_{n+1}-\tau_n)      +\sum_{n=1}^{N_t}[f(X_{\tau_n})-f(X_{\tau_n}^-)] \\
          &-\sum_{n=0}^{N_t}\A f(X_{\tau_n},t\land\tau_{n+1}-\tau_n).
  \end{align*}
   Meantime, both $Df$ and $\A f$ are additive functional of $\phi$ in $\mathfrak{V}_{\phi}^{loc}$. According to Theorem \ref{thm.A=a+b}, $M^f$ is an additive functional of $X$ with
  \begin{align*}
    & a_M(x,t)=Df(x,t)-\A f(x,t),\\
    & b_M(x,t,y)=f(y)-f(\phi(t,x)).
  \end{align*}
  By Theorem \ref{thm.semimartingale} (ii), since the additive functional $M^f$ is a local martingale prior to $\tau$, then we have (\ref{eq.domain.A-general}) and
  $$a_M(x,t)=-\int_{(0,t]}\int_E\big[f(y)-f(\phi(s,x))\big]q(x,s,\ud y)\Lambda(x,\ud s),$$
  which is equivalent to (\ref{eq.measure-valued generator}).

  Conversely, if (\ref{eq.domain.A-general}) is satisfied with $f\in\mathfrak{V}_{\phi}^{loc}$, then $\A f\in\mathfrak{A}_{\phi}^{loc}$. Furthermore, by Lemma \ref{lem.f-f is A}, we have
  \begin{align*}
    M^f_t=&f(X_t)-f(X_0)-\sum_{n=0}^{N_t}\A f(X_{\tau_n},t\land\tau_{n+1}-\tau_n)\\
         =&-\sum_{n=0}^{N_t}\int_{(\tau_n,t\land\tau_{n+1}]}\int_E[f(y)-f(X_s^-)] q(X_{\tau_n},s-\tau_n,\ud y)\Lambda(X_{\tau_n},\ud s-\tau_n)\\
          &+\sum_{n=1}^{N_t}[f(X_{\tau_n})-f(X_{\tau_n}^-)],\qquad\qquad t<\tau.
  \end{align*}
  By Theorem \ref{thm.semimartingale} (ii), it follows from (\ref{eq.domain.A-general}) that $M^f$ is a local martingale prior to $\tau$. Hence $f\in\D(\A)$ and then $(\A,\D(\A))$ is the measure-valued generator of $X$.
\end{proof}

Based on Theorem \ref{thm.measure-valued generator}, we give the It\^o formula for a general PDMP.

\begin{theorem}[\bf The It\^o formula]
  Let $X$ be a general PDMP. If $f\in\D(\A)$, then for $t\in[0,\tau)$,
  \begin{align}\label{eq.Ito's formula}
    f(X_t)-f(X_0)=&\sum_{n=0}^{N_t}\A f(X_{\tau_n},t\land\tau_{n+1}-\tau_n)\nonumber\\
                 &+\int_{(0,t]\times E}\!\!\!\big[f(y)-f(X_s^-)\big](\mu-\nu)(\ud s,\ud y),
  \end{align}
  where the last item on the right side is a local martingale prior to $\tau$. Especially, if
  \begin{equation}\label{eq.martingale condition}
    \Ep_x\Big[\sum_{n=1}^\infty\big|f(X_{\tau_n})-f(X_{\tau_n}^-)\big|\Big]<\infty,\quad x\in E,
  \end{equation}
  then the last item on the right side of \textnormal{(\ref{eq.Ito's formula})} is a martingale.
\end{theorem}

\begin{proof}
  For $f\in\D(\A)$, by (\ref{eq.generator.def}),
  $$f(X_t)-f(X_0)=\sum_{n=0}^{N_t}\A f(X_{\tau_n},t\land\tau_{n+1}-\tau_n)+M^f_t,\quad 0\leqslant t<\tau,$$
  where $M^f$ is a local martingale prior to $\tau$. And, by Theorem \ref{thm.semimartingale} (ii),
  $$M_t^f=\int_{(0,t]\times E}\!\!\!\big[f(y)-f(X_s^-)\big](\mu-\nu)(\ud s,\ud y).$$
  Especially, if (\ref{eq.martingale condition}) holds, then, by \cite[Theorem 26.12]{davis1993markov}, $M^f$ is a martingale. And the proof is complete.
\end{proof}

\subsection{Extended generator and L-extended generator}

In the followings we intend by virtue of the measure-valued generator to extend the domain of two familiar extensions of generator: the extended generator introduced to PDMPs by Davis \cite{davis1984piecewise} and the L-extended generator introduced by Kunita \cite{Kunita1969Absolute}. And we will present their connections with the measure-valued generator. For convenience, we firstly give the classification of measurable functions defined on $\bar{E}$.

\begin{definition}
  A measurable function $f:\bar{E}\mapsto\R$ is called to be
  \begin{enumerate}[(i)]
    \item \emph{path-continuous} if $f(\phi(\cdot,x))$ is continuous on $\mathcal{I}_x$ for all $x\in E$;
    \item \emph{absolutely path-continuous} if $f(\phi(\cdot,x))$ is absolutely continuous on $\mathcal{I}_x$ for all $x\in E$;
    \item \emph{piecewise absolutely path-continuous} if the continuous part of the function $f(\phi(\cdot,x))$ is absolutely continuous on $\mathcal{I}_x$ for all $x\in E$;
    \item \emph{path-step} if $f(\phi(\cdot,x))$ is a step function on $\mathcal{I}_x$ for all $x\in E$.
  \end{enumerate}
\end{definition}

\subsubsection{Extended generator}

\begin{definition}\label{def.extended generator}
  Let $X=\{X_t\}_{0\leqslant t<\tau}$ be a general PDMP with characteristic triple $(\phi,\Lambda,Q)$. $\D(\A')$ denotes the set of $f\in\mathfrak{V}_{\phi}^{loc}$ with the following property: there exist a measurable function $h:E\mapsto\R$ such that 
  $$\int_0^t\big|h(X_s)\big|\ud s<\infty \quad\hbox{for }t\in[0,\tau)\quad\Prb_x\hbox{-a.s. for each }x\in E,$$
  and the process
  \begin{equation*}
    M'_t:=f(X_t)-f(X_0)-\int_0^th(X_s)\ud s,\quad 0\leqslant t<\tau
  \end{equation*}
  is a local martingale prior to $\tau$. Then we write $h=\A'f$. It is well known that $(\A',\D(\A'))$ is called the \emph{extended generator} of the process $X$ (see \cite{davis1993markov} Definition 14.15).
\end{definition}

\begin{theorem}\label{thm.A'.general}
  Let $X$ be a general PDMP with characteristic triple $(\phi,\Lambda,Q)$. Then the domain $\D(\A')$ of the extended generator $\A'$ of $X$ consists of all $f\in\mathfrak{V}_{\phi}^{loc}$
  such that for any $x\in E$, $t\in\mathcal{I}_x$,
  \begin{equation}\label{eq.domain.A'-general.Q}
    \int_{(0,t]}\int_E\big|f(y)-f(\phi(s,x))\big|Q(\phi(s,x),\ud y)\Lambda(x,\ud s)<\infty,
  \end{equation}
  with constraint conditions
  \begin{align}
    &D^{sc}f(x,t)+\!\int_{(0,t]}\int_E\big[f(y)-f(\phi(s,x))\big]Q(\phi(s,x),\ud y)\Lambda^{sc}(x,\ud s)=0,\label{eq.A^sc=0.general}\\
    &D^{pd}f(x,t)+\!\int_{(0,t]}\int_E\big[f(y)-f(\phi(s,x))\big]Q(\phi(s,x),\ud y)\Lambda^{pd}(x,\ud s)=0. \label{eq.A^pd=0.general}
  \end{align}
  For $f\in\D(\A')$, $\A'f$ is given by
  \begin{equation}\label{eq.A'f.general}
    \A'f(x)=\mathcal{X}f(x)+\lambda(x)\int_E[f(y)-f(x)]Q(x,\ud y),\quad x\in E,
  \end{equation}
  where
  $$\mathcal{X}f(x)=\left\{\begin{array}{ll}
       \displaystyle\frac{\partial^+f(\phi(t,x))}{\partial t}\bigg|_{t=0}, &
        \hbox{if }\displaystyle\frac{\partial^+f(\phi(t,x))}{\partial t}\bigg|_{t=0}
       \hbox{ exists;} \\
       0, & \hbox{otherwise,}
     \end{array}\right.$$
  and
  $$\lambda(x)=\left\{\begin{array}{ll}
        \displaystyle \frac{\partial^+\Lambda(x,t)}{\partial t}\bigg|_{t=0}, &
        \hbox{if }\displaystyle\frac{\partial^+\Lambda(x,t)}{\partial t}\bigg|_{t=0}\hbox{ exists;} \\
        0, & \hbox{otherwise.}
      \end{array}\right.$$
\end{theorem}

\begin{proof}
  Suppose that $f\in\D(\A')$. Then
  \begin{align*}
  M'_t&=f(X_t)-f(X_0)-\int_0^t\A'f(X_s)\ud s\\
      &=f(X_t)-f(X_0)-\sum_{n=0}^{N_t}\int_{\tau_n}^{t\land\tau_{n+1}}\A' f(\phi(s-\tau_n,X_{\tau_n}))\ud s,\quad t<\tau
  \end{align*}
  is a local martingale prior to $\tau$. Moreover, by Lemma \ref{lem.f-f is A},
  \begin{align*}
  M'_t=&\sum_{n=0}^{N_t}Df(X_{\tau_n},t\land\tau_{n+1}-\tau_n)
         +\sum_{n=1}^{N_t}[f(X_{\tau_n})-f(X_{\tau_n}^-)]\\
       &-\sum_{n=0}^{N_t}\int_{\tau_n}^{t\land\tau_{n+1}}\A' f(\phi(s-\tau_n,X_{\tau_n}))\ud s, \quad t<\tau.
  \end{align*}
   Meantime, both $Df(x,t)$ and $\int_0^t\A' f(\phi(s,x))\ud s$ belong to $\mathfrak{A}_{\phi}^{loc}$. According to Theorem \ref{thm.A=a+b}, $M'$ is an additive functional of $X$ with
  \begin{align*}
    &a_{M'}(x,t)=Df(x,t)-\int_0^t\A' f(\phi(s,x))\ud s, \\
    &b_{M'}(x,t,y)=f(y)-f(\phi(t,x)). 
  \end{align*}
  Since the additive functional $M'$ is a local martingale prior to $\tau$, then we have (\ref{eq.domain.A'-general.Q}) by Theorem \ref{thm.semimartingale} (ii) and
  $$a_{M'}(x,t)=-\int_{(0,t]}\int_E\big[f(y)-f(\phi(s,x))\big]Q(\phi(s,x),\ud y)\Lambda(x,\ud s),$$
  that is,
  \begin{equation}\label{eq.Af=intA'f}
  \A f(x,t)=\int_0^t\A' f(\phi(s,x))\ud s\quad\hbox{ for all }x\in E,\, t\in\mathcal{I}_x.
  \end{equation}
  Consider the Lebesgue decomposition of $\A f=\A^{ac}f+\A^{sc}f+\A^{pd}f$. The the equation (\ref{eq.Af=intA'f}) is equivalent to (\ref{eq.A'f.general}) with (\ref{eq.A^sc=0.general}) and (\ref{eq.A^pd=0.general}).

  Conversely, if $f\in\mathfrak{V}_{\phi}^{loc}$, and (\ref{eq.domain.A'-general.Q}) with the constraint conditions (\ref{eq.A^sc=0.general}) and (\ref{eq.A^pd=0.general}) are satisfied, then $f\in \D(\A)$. Hence,
  $$M^f_t=f(X_t)-f(X_0)-\sum_{n=0}^{N_t}\A f(X_{\tau_n},t\land\tau_{n+1}-\tau_n),\qquad t<\tau$$
  is a local martingale prior to $\tau$. But, in this case, we have (\ref{eq.Af=intA'f}) which implies that
  $$M^f_t=f(X_t)-f(X_0)-\int_0^t\A'f(X_s)\ud s,\quad t<\tau$$
  is a local martingale prior to $\tau$. Hence $f\in\D(\A')$ and then $(\A',\D(\A'))$ is the extended generator of $X$.
\end{proof}

\begin{corollary}
  Let $X$ be a quasi-Hunt PDMP with characteristic triple $(\phi,\Lambda,Q)$. Then the domain $\D(\A')$ of the extended generator $\A'$ of $X$ consists of each path-continuous function $f\in\mathfrak{V}_{\phi}^{loc}$
  such that for any $x\in E$, $t\in\mathcal{I}_x$,
  \begin{equation}\label{eq.domain.A'-Hunt}
    \int_{(0,t]}\int_E\big|f(y)-f(\phi(s,x))\big|Q(\phi(s,x),\ud y)\Lambda(x,\ud s)<\infty,
  \end{equation}
  with constraint condition
  \begin{equation}\label{eq.A^sc=0.Hunt}
    D^{sc}f(x,t)+\int_{(0,t]}\int_E\big[f(y)-f(\phi(s,x))\big]Q(\phi(s,x),\ud y)\Lambda^{sc}(x,\ud s)=0.
  \end{equation}
  For $f\in\D(\A')$, $\A'f$ is given by
  \begin{equation}\label{eq.A'.Hunt}
    \A'f(x)=\mathcal{X}f(x)+\lambda(x)\int_E[f(y)-f(x)]Q(x,\ud y),\quad x\in E.
  \end{equation}
\end{corollary}

\begin{proof}
 In the case of quasi-Hunt, the discrete part $\Lambda^{pd}=0$ of the characteristic $\Lambda$. The constraint condition (\ref{eq.A^pd=0.general}) becomes $D^{pd}f(x,t)=0$, or equivalently that $f$ is path-continuous. This completes the proof.
\end{proof}

Let
$$\Gamma:=\{\phi(c(x),x):F(x,c(x)-)>0,\,c(x)<\infty,\,x\in E\}.$$
It is obviously that $\Gamma\subset\partial^+E$ and $\Delta\Lambda(x)=1$ for $x\in\Gamma$. Here $\partial^+E$ is the flowing-out boundary of $E$. $\Gamma$ is used to be called the `active boundary' in the PDMPs theory. Davis \cite{davis1993markov} deals essentially with the quasi-It\^o PDMPs apart from possibly certain force jumps at the active boundary. Strictly speaking, a general PDMP $X$ with the characteristic triple $(\phi,\Lambda,Q)$ is called a PDMP in the sense of Davis \cite{davis1993markov} if
\begin{equation}\label{eq.Lambda.Davis}
\Lambda(x,t)=\int_0^t\lambda(\phi(s,x))\ud s+\bbbone_{\{\phi(t,x)\in\Gamma\}},\quad x\in E,\, t\in \mathcal{I}_x.
\end{equation}
Now we represent \cite[Theorem 26.14]{davis1993markov} for the extended generator of PDMPs as a corollary of Theorem \ref{thm.A'.general}.

\begin{corollary}\label{cor.A'.Davis}
  Let $X$ be a general PDMP with characteristic triple $(\phi,\Lambda,Q)$, in which $\Lambda$ satisfies \textnormal{(\ref{eq.Lambda.Davis})}. Then the domain $\D(\A')$ of the extended generator $\A'$ of $X$ consists of all $f\in\mathfrak{V}_{\phi}^{loc}$
  such that for $x\in E$, $f(\phi(\cdot,x))$ is absolutely continuous on $[0,c(x))$, and for any $x\in E$, $t\in\mathcal{I}_x$,
  \begin{equation}\label{eq.domain.A'-Davis}
    \int_0^t\lambda(\phi(s,x))\int_E\big|f(y)-f(\phi(s,x))\big|Q(\phi(s,x),\ud y)\ud s<\infty,
  \end{equation}
  with boundary conditions
  \begin{equation}\label{eq.boundary condition.Davis}
     \Delta f(x)+\int_E\big[f(y)-f(x)\big]Q(x,\ud y)=0,\quad x\in \Gamma.
  \end{equation}
  For $f\in\D(\A')$, $\A'f$ is given by
  \begin{equation}\label{eq.A'.Davis}
    \A'f(x)=\mathcal{X}f(x)+\lambda(x)\int_E[f(y)-f(x)]Q(x,\ud y),\quad x\in E.
  \end{equation}
\end{corollary}

\begin{proof}
  In view of (\ref{eq.Lambda.Davis}), the constraint conditions (\ref{eq.A^sc=0.general}) and (\ref{eq.A^pd=0.general}) become that $f(\phi(\cdot,x))$ is absolutely continuous on $[0,c(x))$ for $x\in E$ with the boundary condition (\ref{eq.boundary condition.Davis}). In this case, (\ref{eq.domain.A'-general.Q}) is equivalent to (\ref{eq.domain.A'-Davis}).
\end{proof}

Compare Corollary \ref{cor.A'.Davis} with Theorem 26.14 in Davis \cite{davis1993markov}, there are two differences for $\D(\A')$. Firstly, in Theorem 26.14 of \cite{davis1993markov}, $f\in\D(\A')$ satisfies
$$\Ep_x\left[\int_{\R_+\times E}\big[f(y)-f(X_{s-})\big]\bbbone_{\{t<T_n\}}\nu(\ud s,\ud y)\right]<\infty$$
for some sequence of stopping times $\{T_n\}$ with $T_n\uparrow\infty$ a.s. instead of (\ref{eq.domain.A'-Davis}). While, (\ref{eq.domain.A'-Davis}) is an analytic condition, which is easier to be verified. Secondly, in Theorem 26.14 of \cite{davis1993markov}, the boundary condition is
$$f(x)=\int_Ef(y)Q(x,\ud y),\quad x\in\Gamma.$$
But, in Corollary \ref{cor.A'.Davis}, we allow $\Delta f(x)\neq 0$ with (\ref{eq.boundary condition.Davis}). Thus, we get a larger domain of extended generator for the PDMPs in the sense of Davis \cite{davis1993markov}.

\subsubsection{L-extended generator}

Consider a basic additive functional $L$ of the general PDMP $X$ defined as
$$L_0=0,\quad L_t=L_{\tau_n}+\Lambda(X_{\tau_n},t-\tau_n) \quad \hbox{for }t\in(\tau_n,\tau_{n+1}],\,n\in\N.$$

\begin{definition}\label{def.L-extended generator}
  Let $X=\{X_t\}_{0\leqslant t<\tau}$ be a general PDMP with characteristic triple $(\phi,\Lambda,Q)$. $\D(\A'')$ denotes the set of $f\in\mathfrak{V}_{\phi}^{loc}$ with the following property: there exist a measurable function $h:E\mapsto\R$ such that
  $$\int_{(0,t]}|h(X^-_s)|\ud L_s<\infty\quad\hbox{for }t\in[0,\tau)\quad\Prb_x\hbox{-a.s. for each }x\in E$$
  and the process
  \begin{equation*}
    M''_t:=f(X_t)-f(X_0)-\int_{(0,t]}h(X^-_s)\ud L_s,\quad 0\leqslant t<\tau,
  \end{equation*}
  is a local martingale prior to $\tau$. Then we write $h=\A''f$. $(\A'',\D(\A''))$ is called the \emph{L-extended generator} of the process $X$.
\end{definition}

\begin{theorem}
  Let $X$ be a general PDMP with characteristic triple $(\phi,\Lambda,Q)$. Then the domain $\D(\A'')$ of the L-extended generator $\A''$ of $X$ consists of all $f\in\mathfrak{V}_{\phi}^{loc}$ such that for any $x\in E$, $t\in\mathcal{I}_x$,
  \begin{equation}\label{eq.domain.LA"-general}
    \int_{(0,t]}\int_E\big|f(y)-f(\phi(s,x))\big|Q(\phi(s,x),\ud y)\Lambda(x,\ud s)<\infty,
  \end{equation}
  and there exists a measurable function $Kf(\cdot)$ such that
  \begin{equation}\label{eq.K-condition.general}
    Df(x,t)=\int_{(0,t]}Kf(\phi(s,x))\Lambda(x,\ud s),\quad t\in\mathcal{I}_x,\,x\in E.
  \end{equation}
  For $f\in\D(\A'')$, $\A''f$ is given by
  \begin{equation}\label{eq.LA".general}
    \A''f(x)=Kf(x)+\int_E[f(y)-f(x)]Q(x,\ud y),\quad x\in \bar{E}.
  \end{equation}
\end{theorem}

\begin{proof}
  Suppose that $f\in\D(\A'')$. Hence
  \begin{align*}
    M''_t&=f(X_t)-f(X_0)-\int_{(0,t]}\A''f(X^-_s)\ud L_s\\
         &=f(X_t)-f(X_0)-\sum_{n=0}^{N_t}\int_{(\tau_n,t\land\tau_{n+1}]}\!\!\!\A'' f(\phi(s-\tau_n,X_{\tau_n}))\Lambda(X_{\tau_n},\ud s),\quad t<\tau
  \end{align*}
  is a local martingale prior to $\tau$. Moreover, by Lemma \ref{lem.f-f is A},
  \begin{align*}
    M''_t=&\sum_{n=0}^{N_t}Df(X_{\tau_n},t\land\tau_{n+1}-\tau_n)     +\sum_{n=1}^{N_t}[f(X_{\tau_n})-f(X_{\tau_n}^-)]\\
          &-\sum_{n=0}^{N_t}\int_{(\tau_n,t\land\tau_{n+1}]}\A''f (\phi(s-\tau_n,X_{\tau_n}))\Lambda(X_{\tau_n},\ud s),\quad t<\tau.
  \end{align*}
  Meantime, both $Df(x,t)$ and $\int_{(0,t]}\A''f(\phi(s,x))\Lambda(x,\ud s)$ belong to $\mathfrak{A}_{\phi}^{loc}$. According to Theorem \ref{thm.A=a+b}, $M''$ is an additive functional of $X$ with
  \begin{align*}
    & a_{M''}(x,t)=Df(x,t)-\int_{(0,t]}\A''f(\phi(s,x))\Lambda(X_{\tau_n},\ud s),\\
    & b_{M''}(x,t,y)=f(y)-f(\phi(t,x)).
  \end{align*}
  By Theorem \ref{thm.semimartingale} (ii), since the additive functional $M''$ is a local martingale prior to $\tau$, then we have (\ref{eq.domain.LA"-general}) and
  $$a_{M''}(x,t)=-\int_{(0,t]}\int_E\big[f(y)-f(\phi(s,x))\big]Q(\phi(s,x),\ud y)\Lambda(x,\ud s),$$
  that is,
  \begin{equation}\label{eq.Af=intA"f}
    \A f(x,t)=\int_{(0,t]}\A'' f(\phi(s,x))\Lambda(x,\ud s)\quad\hbox{ for all }x\in E,\, t\in\mathcal{I}_x.
  \end{equation}
  Equivalently,
  $$Df(x,t)=\int_{(0,t]}\!\left[\A''f(\phi(s,x))-\int_E[f(y)-f(\phi(s,x)]Q(\phi(s,x),\ud y)\right]\!\Lambda(x,\ud s)$$
  for all $x\in E$, $t\in\mathcal{I}_x$.
  It follows $Df(x,t)=\int_{(0,t]}Kf(\phi(s,x))\Lambda(x,\ud s)$ with
  $$Kf(x)=\A''f(x)-\int_E[f(y)-f(x)]Q(\phi(s,x),\ud y),\quad x\in \bar{E},$$
  which is (\ref{eq.LA".general}).

  Conversely, if $f\in\mathfrak{V}_{\phi}^{loc}$ satisfies (\ref{eq.domain.LA"-general}), then $f\in \D(\A)$. Hence,
  $$M^f_t=f(X_t)-f(X_0)-\sum_{n=0}^{N_t}\A f(X_{\tau_n},t\land\tau_{n+1}-\tau_n),\quad t<\tau$$
  is a local martingale prior to $\tau$. But, in this case, we have (\ref{eq.Af=intA"f}) which implies
  $$M^f_t=f(X_t)-f(X_0)-\int_0^t\A''f(X_s)\ud L_s,\quad t<\tau$$
  is a local martingale prior to $\tau$. Hence $f\in\D(\A'')$ and then $(\A'',\D(\A''))$ is the L-extended generator of $X$.
\end{proof}

\begin{corollary}
  Let $X$ be a quasi-step PDMP with characteristic triple $(\phi,\Lambda,Q)$. Then the domain $\D(\A'')$ of the L-extended generator $\A''$ of $X$ consists of each path-step function $f\in\mathfrak{V}_{\phi}^{loc}$ with $J_{Df}\subseteq J_\Lambda$ such that for any $x\in E$, $t\in\mathcal{I}_x$,
  \begin{equation}\label{eq.domain.LA"-step}
    \sum_{0<s\leqslant t}\Delta\Lambda(x,\ud s)\int_E\big|f(y)-f(\phi(s,x))\big|Q(\phi(s,x),\ud y)<\infty,
  \end{equation}
  and there exists a measurable function $Kf(\cdot)$ such that
  \begin{equation}\label{eq.K-condition.step}
  \Delta f(x)=Kf(x)\Delta\Lambda(x),\quad x\in J_\Lambda.
  \end{equation}
  For $f\in\D(\A'')$, $\A''f$ is given by
  \begin{equation}\label{eq.LA".step}
    \A''f(x)=Kf(x)+\int_E[f(y)-f(x)]Q(x,\ud y),\quad x\in J_\Lambda.
  \end{equation}
\end{corollary}

\begin{proof}
  In the case of quasi-step, the parts $\Lambda^{ac}=0$ and $\Lambda^{sc}=0$ of the characteristic $\Lambda$. The condition (\ref{eq.K-condition.general}) becomes that
  $$Df(x,t)=\sum_{0<s\leqslant t}Kf(\phi(s,x))\Delta\Lambda(\phi(s,x)),\quad t\in\mathcal{I}_x,\,x\in E,$$
  or equivalently, that $f$ is a path-step function with $J_{Df}\subseteq J_\Lambda$ and (\ref{eq.K-condition.step}). This completes the proof.
\end{proof}

\section{Expectations of additive functionals and measure differential equations for a general PDMP}

This section deals with calculating the expected cumulative discounted value of an additive functional of a general PDMP,
\begin{equation}\label{eq.V.general}
  V(x):=\Ep_x\left[\int_{(0,\tau]}e^{-\delta s}\ud A_s\right],\quad x\in E,
\end{equation}
where the discount rate $\delta>0$. This functional appears as a `value function' (or, `reward function', `cost function', etc.) in stochastic service systems, stochastic control theory, and is sufficiently general to cover a wide variety of apparently different functionals as special cases.

\begin{theorem}\label{thm.V.expectation.predictable}
  Let $X$ be a general PDMP. Assume that $A=\{A_t\}_{0\leqslant t<\tau}$ is a predictable additive functional of $X$ with the associated $a\in\mathfrak{A}_{\phi}^{loc}$ .
  If $V\in\D(\A)$ defined as \textnormal{(\ref{eq.V.general})}, then it satisfies the following measure integro-differential equation
  \begin{equation}\label{eq.V.equation}
    \A V(x,\ud t)+a(x,\ud t)=\delta V(\phi(t,x))\ud t, \quad t\in\mathcal{I}_x,\,x\in E.
  \end{equation}
\end{theorem}

\begin{proof}
  Note that if $A$ is predictable, for any fixed $x\in E$, by the strong Markov property, it follows from Corollary \ref{cor.A.pred=a} that
  \begin{align*}
    V(x) =& \Ep_x\left[\int_{(0,t\land\tau_1]}e^{-\delta s}a(x,\ud s) + e^{-\delta(t\land\tau_1)}V(X_{t\land\tau_1})\right] \\
         =& F(x,t)\int_{(0,t]}e^{-\delta s}a(x,\ud s) +\int_{(0,t]}\int_{(0,s]}e^{-\delta u}a(x,\ud u)F(x,\ud s)  \\
          &  +e^{-\delta t}V(\phi(t,x))F(x,t) +\int_{(0,t]}e^{-\delta s}\int_EV(y)q(x,s,\ud y)F(x,\ud s).
  \end{align*}
  By the formula for integration by parts,
  \begin{align*}
   & F(x,t)\int_{(0,t]}e^{-\delta s}a(x,\ud s) \\
  =& -\int_{(0,t]}\int_{(0,s]}e^{-\delta u}a(x,\ud u)F(x,\ud s)+\int_{(0,t]}e^{-\delta s}F(x,s-)a(x,\ud s) .
  \end{align*}
  Thus,
  $$V(x) = \int_{(0,t]}\!\!e^{-\delta s}F(x,s-)H(x,\ud s) +e^{-\delta t} V(\phi(t,x))F(x,t),$$
  which is equivalent to
  $$e^{-\delta t}V(\phi(t,x))=\frac{V(x)}{F(x,t)}-\frac{1}{F(x,t)}\int_{(0,t]}\!\!e^{-\delta s}F(x,s-)H(x,\ud s),$$
  where $H(x,\ud s):=a(x,\ud s)+\Lambda(x,\ud s)\int_EV(y)q(x,s,\ud y)$. Substitute the equation above into the following integration
  \begin{align*}
     & \int_{(0,t]}e^{-\delta s}V(\phi(s,x))\Lambda(x,\ud s) \\
    =& \int_{(0,t]}\left[\frac{V(x)}{F(x,s)}-\frac{1}{F(x,s)} \int_{(0,s]}\!\!e^{-\delta u} F(x,u-)H(x,\ud u)\right] \Lambda(x,\ud s) \\
    =& V(x)\int_{(0,t]}\frac{\Lambda(x,\ud s)}{F(x,s)}-\int_{(0,t]}e^{-\delta s}F(x,s-) \int_{[s,t]}\frac{\Lambda(x,\ud u)}{F(x,u)}H(x,\ud s) \\
    =& \frac{V(x)}{F(x,t)}-V(x)-\frac{1}{F(x,t)}\int_{(0,t]}e^{-\delta s}F(x,s-)H(x,\ud s)+\int_{(0,t]}e^{-\delta s}H(x,\ud s) \\
    =& e^{-\delta t}V(\phi(t,x))-V(x)+\int_{(0,t]}e^{-\delta s}H(x,\ud s).
  \end{align*}
  On the other hand, by the formula for integration by parts,
  $$e^{-\delta t}V(\phi(t,x))-V(x)=\int_{(0,t]}e^{-\delta s}DV(x,\ud s)-\int_0^te^{-\delta s}\delta V(\phi(s,x))\ud s.$$
  Thus, we have
  \begin{align*}
    &\int_{(0,t]}e^{-\delta s}V(\phi(s,x))\Lambda(x,\ud s)\\
   =&\int_{(0,t]}e^{-\delta s}DV(x,\ud s)-\int_0^te^{-\delta s}\delta V(\phi(s,x))\ud s  +\int_{(0,t]}e^{-\delta s}H(x,\ud s)
  \end{align*}
  for all $t\in\mathcal{I}_x$. Noticing that the both sides of the equation above are additive functionals of the SDS $\phi$, we have
  $$V(\phi(t,x))\Lambda(x,\ud t)=DV(x,\ud t)-\delta V(\phi(t,x))\ud t+H(x,\ud t)$$
  for all $t\in\mathcal{I}_x$, which is equivalent to (\ref{eq.V.equation}) by the arbitrary of $x$.
\end{proof}

\begin{corollary}\label{cor.V.expectation.optional}
   Assume that $A$ is an optional additive functional of $X$, the associated function $a\in\mathfrak{A}_\phi^{loc}$ and $b$ satisfies
   \begin{equation}\label{eq.b.martingale}
     \Ep_x\left[\sum_{n=1}^{\infty}e^{-\delta\tau_n} \big|b(X_{\tau_{n-1}},\tau_n-\tau_{n-1},X_{\tau_n})\big|\right]<\infty,\quad x\in E.
   \end{equation}
   If $V\in\D(\A)$, then $V$ satisfies \textnormal{(\ref{eq.V.equation})} by replacing $a$ in Theorem \textnormal{\ref{thm.V.expectation.predictable}} with
  \begin{equation*}\label{eq.V.equation.optional}
    a^*(x,\ud t)=a(x,\ud t)+\Lambda(x,\ud t) \int_Eb(x,t,y)q(x,t,\ud y),\quad x\in E.
  \end{equation*}
\end{corollary}

\begin{proof}
  If $A$ is optional and associated with $a$ and $b$, then
  \begin{equation*}
    V(x)\! = \!\Ep_x\!\left[\sum_{n=0}^{\infty}\!\left(\int_{(\tau_n,\tau_{n+1}]}\!\!\!\!\!\!\!\!\!e^{-\delta s}a(X_{\tau_n},\ud s-\tau_n) +e^{-\delta\tau_{n+1}}b(X_{\tau_n},\tau_{n+1}-\tau_n,X_{\tau_{n+1}})\right)\right].
  \end{equation*}
  Notice that (\ref{eq.b.martingale}) is satisfied,
  \begin{align*}
    &\Ep_x\left[\sum_{n=0}^{\infty}e^{-\delta\tau_{n+1}} b(X_{\tau_n},\tau_{n+1}-\tau_n,X_{\tau_{n+1}})\right]\\
   =&\Ep_x\left[\sum_{n=0}^{\infty}\int_{(\tau_n,\tau_{n+1}]\times E} \!\! e^{-\delta s} b(X_{\tau_n},s-\tau_n,y)\mu(\ud s,\ud y)\right]\\
   =&\Ep_x\left[\sum_{n=0}^{\infty}\int_{(\tau_n,\tau_{n+1}]\times E} \!\! e^{-\delta s} b(X_{\tau_n},s-\tau_n,y)\nu(\ud s,\ud y)\right]\\
   =&\Ep_x\left[\sum_{n=0}^{\infty}\int_{(\tau_n,\tau_{n+1}]}\!\!\!\!\! e^{-\delta s} \int_Eb(X_{\tau_n},s-\tau_n,y)q(X_{\tau_n},s-\tau_n,\ud y)
     \Lambda(X_{\tau_n},\ud s-\tau_n)\right],
  \end{align*}
  and $\int_{(0,t]}\int_Eb(x,s,y)q(x,s,\ud y)\Lambda(x,\ud s)\in\mathfrak{A}_\phi^{loc}$. Hence, $a^*\in\mathfrak{A}_\phi^{loc}.$
  Define a process $A^*$ by $A^*_0=0$ and
  $$A^*_t=A^*_{\tau_n}+a^*(X_{\tau_n},t-\tau_n),\quad t\in(\tau_n,\tau_{n+1}]\,n\in\N.$$
  Thus
  $$V(x)=\Ep_x\left[\int_{(0,\tau]}e^{-\delta s}\ud A^*_s\right],\quad x\in E.$$
  Then, applying Theorem \ref{thm.V.expectation.predictable}, we prove the corollary.
\end{proof}


Theorem \ref{thm.V.expectation.predictable} tells us that there exists one solution $V$ of the equation (\ref{eq.V.equation}). And we would like to know that the equation has a unique solution. The following is a general result in this direction.

\begin{theorem}\label{thm.UniqueSolution}
  Suppose a measurable function $f\in\D(\A)$ satisfies the measure integro-differential equation
  \begin{equation}\label{eq.f.eq.general}
    \A f(x,\ud t)+a(x,\ud t)=\delta f(\phi(t,x))\ud t, \quad t\in\mathcal{I}_x,\,x\in E,
  \end{equation}
  where $a\in\mathfrak{A}^{loc}_{\phi}$ is associated with the predictable additive functional $A$ of $X$. Suppose further that $f$ satisfies that
  \begin{equation}\label{eq.m-condition.general}
    \Ep_x\left[\sum_{n=1}^{\infty}e^{-\delta\tau_n} \big|f(X_{\tau_n})-f(X_{\tau_n}^-)\big|\right]<\infty,   \quad x\in E,
  \end{equation}
  and that $\Ep_x\big[e^{-\delta t}f(X_{t\land\tau})\big]\to 0$ as $t\to\infty$ for all $x\in E$.
  Then
  \begin{equation}\label{eq.unique solution}
    f(x)=\Ep_x\left[\int_{(0,\tau]}e^{-\delta s}\ud A_s\right],\quad x\in E.
  \end{equation}
\end{theorem}

\begin{proof}
  Since $f\in\D(\A)$, applying the formula for integration by parts and the It\^o formula (\ref{eq.Ito's formula}), we get
  \begin{align*}
   & e^{-\delta t}f(X_t)-f(X_0)\\
  =& \int_{(0,t]}e^{-\delta s}\ud f(X_s)-\int_0^te^{-\delta s}\delta f(X_s)\ud s \\
  =& \sum_{n=0}^{N_t}\int_{(\tau_n,t\land\tau_{n+1}]}\!\!\!e^{-\delta s}\A f(X_{\tau_n},\ud s-\tau_n)-\int_0^te^{-\delta s}\delta f(X_s)\ud s\\
   &+\int_{(0,t]\times E}\!\!e^{-\delta s}\big[f(y)-f(X_s^-)\big](\mu-\nu)(\ud s,\ud y)
  \end{align*}
  for any $t\in[0,\tau)$.
  Since $f$ satisfies (\ref{eq.m-condition.general}), the last term above is a martingale with mean zero. In view of $\A f(x,\ud t)=-a(x,\ud t)+\delta f(\phi(t,x))\ud t$, we get by taking expectations
  \begin{align*}
    f(x)=&\Ep_x\big[e^{-\delta t}f(X_{t\land\tau})\big]
         +\Ep_x\left[\sum_{n=0}^{N_t}\int_{(\tau_n,t\land\tau_{n+1}]}\!\!\!e^{-\delta s}a(X_{\tau_n},\ud s)\right]\\
        =&\Ep_x\big[e^{-\delta t}f(X_{t\land\tau})\big]+\Ep_x\left[\int_{(0,t]}e^{-\delta s}\ud A_s\right].
  \end{align*}
  Thus (\ref{eq.unique solution}) follows by the assumption that $\lim_{t\to\infty}\Ep_x\big[e^{-\delta t}f(X_{t\land\tau})\big]=0$.
\end{proof}


\begin{theorem}\label{thm.equation.Lebesgue}
  Let $a\in\mathfrak{A}_\phi^{loc}$. Assume that $J_\Lambda\cup J_a$ contains no confluent state.
  Function $f\in\D(\A)$ satisfies the measure integro-differential equation
  \begin{equation}\label{eq.measure.diff.equation}
    \A f(x,\ud t)+a(x,\ud t)=\delta f(\phi(t,x))\ud t,\quad t\in\mathcal{I}_x,\,x\in E
  \end{equation}
  if and only if it satisfies the equations
  \begin{numcases}
    \displaystyle\int_0^t\mathcal{X}\A f(\phi(s,x))\ud s+\int_0^t\mathcal{X}a(\phi(s,x))\ud s=\int_0^t\delta f(\phi(s,x))\ud s,\label{eq.f.eq.ac}\\
    \displaystyle\sum_{0<s\leqslant t}\Delta\A f(\phi(s,x))+\sum_{0<s\leqslant t}\Delta a(\phi(s,x))=0,\label{eq.f.eq.pd}\\
    \A^{sc}f(x,t)+a^{sc}(x,t)=0\label{eq.f.eq.sc}
  \end{numcases}
  for $t\in\mathcal{I}_x$, $x\in E$.
\end{theorem}

\begin{proof}
  Notice that the measure integro-differential equation (\ref{eq.measure.diff.equation}) is equivalent to
  \begin{equation}\label{eq.f.eq.add}
    \A f(x,t)+a(x,t)=\int_0^t\delta f(\phi(s,x))\ud s,\quad t\in\mathcal{I}_x,\,x\in E.
  \end{equation}
  The both sides of the equation (\ref{eq.f.eq.add}) are additive functionals of the SDS $\phi$, and the right side is absolutely continuous on $\mathcal{I}_x$ for all $x\in E$.
  Since $J_\Lambda\cup J_a$ contains no confluent state, following from Theorem \ref{thm.a.Lebesgue}, we have the Lebesgue decompositions of $\A f$ and $a$. And, the absolutely continuous parts, the singularly continuous parts and the purely discontinuous parts of the both sides are equal, respectively.

  Conversely, suppose that the equations (\ref{eq.f.eq.ac})-(\ref{eq.f.eq.sc}) hold. Summing up the both sides of the equations, we get (\ref{eq.f.eq.add}) which is equivalent to the measure integro-differential equation (\ref{eq.measure.diff.equation}).
\end{proof}

Consider the absolutely continuous part of the equations above, i.e., equation (\ref{eq.f.eq.ac}), which is equivalent to
\begin{equation*}
  \mathcal{X}\A f(\phi(t,x))+\mathcal{X}a(\phi(t,x))=\delta f(\phi(t,x))\quad\hbox{a.e. on }\mathcal{I}_x \hbox{ for all }x\in E.
\end{equation*}
It implies that, for a solution of the measure integro-differential equation (\ref{eq.measure.diff.equation}), it is not required to satisfy the equation
\begin{equation}\label{eq.f.eq.ac.x}
  \mathcal{X}\A f(x)+\mathcal{X}a(x)=\delta f(x)
\end{equation}
for all $x\in E$. Here we introduce the concept of the so-called path-solution and almost everywhere path-solution.
\begin{definition}
  We call $f$ a \emph{path-solution} of the equation (\ref{eq.f.eq.ac.x}) if it satisfies
  $$\mathcal{X}\A f(\phi(t,x))+\mathcal{X}a(\phi(t,x))=\delta f(\phi(t,x))\quad\hbox{for all } t\in\mathcal{I}_x,x\in E.$$
  And $f$ is called an \emph{almost everywhere path-solution} of (\ref{eq.f.eq.ac.x}) if it satisfies
  $$\mathcal{X}\A f(\phi(t,x))+\mathcal{X}a(\phi(t,x))=\delta f(\phi(t,x))\quad\hbox{a.e. on }\mathcal{I}_x\hbox{ for all }x\in E.$$
  In this sense, we say the equation (\ref{eq.f.eq.ac.x}) holds \emph{path-almost everywhere} (\emph{path-a.e.} in short).
\end{definition}

Considering the purely discontinuous part (\ref{eq.f.eq.pd}), we get
$$\Delta f(x)=\Delta\Lambda(x)\int_E[f(x)-f(y)]Q(x,\ud y)-\Delta a(x),\quad x\in \bar{E}.$$
We observe that $J_{Df}\subseteq J_\Lambda\cup J_a$. Thus, (\ref{eq.f.eq.pd}) is equivalent to the equation
$$\Delta\A f(x)+\Delta a(x)=0,\quad x\in J_\Lambda\cup J_a,$$
with the condition $J_{Df}\subseteq J_\Lambda\cup J_a$.

Following from the singularly continuous part (\ref{eq.f.eq.sc}), we have
$$D^{sc}f(x,t)=\int_{(0,t]}\int_E[f(\phi(s,x))-f(y)]q(x,s,\ud y)\Lambda^{sc}(x,\ud s)-a^{sc}(x,t)$$
for $t\in\mathcal{I}_x$, $x\in E$. Hence, if $\Lambda^{sc}=a^{sc}=0$, then $D^{sc}f=0$, i.e., $f$ is piecewise absolutely path-continuous.


Now we will give three simplified versions of the two theorems above in case of quasi-It\^o, quasi-step and `non-singular', respectively.


\begin{corollary}\label{cor.V.expectation.Ito}
  Let $X$ be a quasi-It\^o PDMP, and $l$ a locally absolutely path-integrable function. Define
  \begin{equation}\label{eq.V.expectation.Ito}
    V(x):=\Ep_x\left[\int_0^\tau e^{-\delta s}l(X_s)\ud s\right],\quad x\in E.
  \end{equation}
  If $V\in\D(\A)$, then it is absolutely path-continuous and satisfies
  \begin{equation}\label{eq.V.eq.Ito}
    \mathcal{X}V(x)+\lambda(x)\int_E[V(y)-V(x)]Q(x,\ud y)+l(x)=\delta V(x)\quad \hbox{path-a.e.}
  \end{equation}
\end{corollary}

\begin{proof}
  Let $A$ be a predictable additive functional of $X$ with the associated $a$ which is defined by
  $$a(x,t)=\int_0^tl(\phi(s,x))\ud s\quad t\in\mathcal{I}_x,\,x\in E.$$
  By the locally absolutely path-integrability of $l$, we know that $a$ is of locally finite variation. Applying Theorem \ref{thm.V.expectation.predictable}, we get that $V$ satisfies the measure integro-differential equation (\ref{eq.V.equation}).
  Thus, by Theorem \ref{thm.equation.Lebesgue}, we get that $V$ is absolutely path-continuous and satisfies (\ref{eq.V.eq.Ito}). The proof is completed.
\end{proof}

\begin{corollary}\label{cor.UniqueSolution.Ito}
  Let $X$ be a quasi-It\^o PDMP. Suppose that $f\in\D(\A)$ is absolutely path-continuous and satisfies
  \begin{equation}\label{eq.f.eq.Ito}
    \mathcal{X}f(x)+\lambda(x)\int_E[f(y)-f(x)]Q(x,\ud y)+l(x)=\delta f(x)\quad \hbox{path-a.e.},
  \end{equation}
  where $l$ is a locally absolutely path-integrable function.
  Suppose further that $f$ satisfies that
  $$\Ep_x\left[\sum_{n=1}^\infty e^{-\delta\tau_n} \big|f(X_{\tau_n})-f(X_{\tau_n}^-)\big|\right]<\infty,\quad x\in E$$
  and that $\Ep_x[e^{-\delta t}f(X_{t\land\tau})]\to 0$ as $t\to\infty$ for all $x\in E$.
  Then
  \begin{equation}\label{eq.unique.sol.Ito}
    f(x)=\Ep_x\left[\int_0^\tau e^{-\delta s}l(X_s)\ud s\right],\quad x\in E.
  \end{equation}
\end{corollary}

\begin{proof}
  For a quasi-It\^o PDMP $X$, $\Lambda$ is absolutely path-continuous. If $f$ is also absolutely path-continuous, then $\A f(x,\cdot)$ is absolutely continuous for all $x\in E$. Thus, following from (\ref{eq.f.eq.Ito}), we get that $f$ satisfies the measure integro-differential equation (\ref{eq.f.eq.general}) with $a(x,t)=\int_0^tl(\phi(s,x))\ud s$. Hence, applying Theorem \ref{thm.UniqueSolution}, we have
  $$f(x)=\Ep_x\left[\int_{(0,t]} e^{-\delta s}\ud A_s\right]=\Ep_x\left[\int_0^\tau e^{-\delta s}l(X_s)\ud s\right],$$
  where $A$ is the predictable additive functional of $X$ corresponding to $a$.
\end{proof}


\begin{corollary}\label{cor.V.expectation.step}
  Let $X$ be a quasi-step PDMP.
  Define
  \begin{equation}
    V(x):=\Ep_x\left[\int_0^te^{-\delta s}l(X_s)\ud s+\sum_{0<s<\tau}e^{-\delta s}g(X_s^-)\right],\quad x\in E,
  \end{equation}
  where $l$ is a locally absolutely path-integrable function, and $g$ is a locally absolutely path-summable function.
  Assume that $J_\Lambda$ contains no confluent state.
  If $V\in\D(\A)$, then it is piecewise absolutely path-continuous with $J_{DV}\subseteq J_\Lambda\cup \mathrm{Supp}_g$ and satisfies
  \begin{equation}\label{eq.V.eq.step}
    \left\{
      \begin{array}{l}
        \displaystyle\mathcal{X}V(x)+l(x)=\delta V(x)\quad\hbox{path-a.e.},\\
        \displaystyle\Delta V(x)+\Delta\Lambda(x)\int_E[V(y)-V(x)]Q(x,\ud y)+g(x)=0,\; x\in J_\Lambda\cup \mathrm{Supp}_g,
      \end{array}
    \right.
  \end{equation}
  where $\mathrm{Supp}_g$ is the support of the function $g$.
\end{corollary}

\begin{proof}
  Let $A$ be a predictable additive functional of $X$ with the associated
  $$a(x,t)=\int_0^tl(\phi(s,x))\ud s+\sum_{0<s\leqslant t}g(\phi(s,x)),\quad t\in\mathcal{I}_x,\,x\in E.$$
  By the locally absolutely path-integrability of $l$ and the locally absolutely path-summability of $g$, we get $a\in\mathfrak{A}_\phi^{loc}$. Applying Theorem \ref{thm.V.expectation.predictable}, we have that $V$ satisfies the measure integro-differential equation (\ref{eq.V.equation}). Noticing that $\Lambda^{ac}=\Lambda^{sc}=a^{sc}=0$, by Theorem \ref{thm.equation.Lebesgue}, we get the conclusion.
\end{proof}

\begin{corollary}\label{cor.UniqueSolution.step}
  Let $X$ be a quasi-step PDMP, $l$ a locally absolutely path-integrable function, and $g$ a locally absolutely path-summable function. Assume that $J_\Lambda$ contains no confluent state. Suppose that $f\in\D(\A)$ is piecewise absolutely path-continuous with $J_{Df}\subseteq J_\Lambda\cup\mathrm{Supp}_g$ and satisfies
  \begin{equation}\label{eq.f.eq.step}
    \left\{
      \begin{array}{l}
        \displaystyle\mathcal{X}f(x)+l(x)=\delta f(x)\quad \hbox{path-a.e.},\\
        \displaystyle\Delta f(x)+\Delta\Lambda(x)\int_E[f(y)-f(x)]Q(x,\ud y)+g(x)=0,\quad x\in J_\Lambda\cup \mathrm{Supp}_g.
      \end{array}
    \right.
  \end{equation}
  Suppose further that $f$ satisfies that
  $$\Ep_x\left[\sum_{n=1}^\infty e^{-\delta\tau_n} \big|f(X_{\tau_n})-f(X_{\tau_n}^-)\big|\right]<\infty,\quad x\in E$$
  and that $\Ep_x[e^{-\delta t}f(X_{t\land\tau})]\to 0$ as $t\to\infty$ for all $x\in E$.
  Then
  \begin{equation}
    f(x)=\Ep_x\left[\int_0^te^{-\delta s}l(X_s)\ud s+\sum_{0<s<\tau}e^{-\delta s}g(X_s^-)\right],\quad x\in E.
  \end{equation}
\end{corollary}

\begin{proof}
  For a quasi-step PDMP $X$, if $f$ is piecewise absolutely path-continuous with $J_{Df}\subseteq J_\Lambda\cup\mathrm{Supp}_g$, then $\A f(x,\cdot)$ is piecewise absolutely continuous for all $x\in E$ with $J_{\A f}\subseteq J_\Lambda\cup\mathrm{Supp}_g$. Thus, it is following from the equations (\ref{eq.f.eq.step}) that $f$ satisfies the measure integro-differential equation (\ref{eq.f.eq.general}) with
  $$a(x,t)=\int_0^tl(\phi(s,x))\ud s+\sum_{0<s\leqslant t}g(\phi(s,x)),\quad t\in\mathcal{I}_x,\,x\in E.$$
  By Theorem \ref{thm.UniqueSolution}, the result is following from
  $$f(x)=\Ep_x\left[\int_{(0,t]} e^{-\delta s}\ud A_s\right],\quad x\in E.$$
  This completes the proof.
\end{proof}


\begin{corollary}\label{cor.V.expectation.nonsingular}
  Let $X$ be a general PDMP with $\Lambda^{sc}=0$. Assume that $J_\Lambda$ contains no confluent state. Define
  \begin{equation}
    V(x):=\Ep_x\left[\int_0^\tau e^{-\delta s}l(X_s)\ud s +\sum_{0<s<\tau} e^{-\delta s} g(X_s^-)\right],\quad x\in E,
  \end{equation}
  where $l$ is a locally absolutely path-integrable function, and $g$ is a locally absolutely path-summable function.
  If $V\in\D(A)$, then it is piecewise absolutely path-continuous with $J_{DV}\subseteq J_\Lambda\cup\mathrm{Supp}_g$ and satisfies
  \begin{equation}\label{eq.V.eq.nonsinglar}
    \left\{
      \begin{array}{l}
        \displaystyle\mathcal{X}V(x)+\lambda(x)\int_E[V(y)-V(x)]Q(x,\ud y)+l(x)=\delta V(x)\quad \hbox{path-a.e.},\\
        \displaystyle\Delta V(x)+\Delta\Lambda(x)\int_E[V(y)-V(x)]Q(x,\ud y)+g(x)=0,\; x\in J_\Lambda\cup\mathrm{Supp}_g.
      \end{array}
    \right.
  \end{equation}
\end{corollary}

\begin{proof}
  Let $A$ be a predictable additive functional of $X$ with the associated
  $$a(x,t)=\int_0^tl(\phi(s,x))\ud s+\sum_{0<s\leqslant t}g(\phi(s,x)),\quad t\in\mathcal{I}_x,\,x\in E.$$
  By the locally absolutely path-integrability and path-summability of $l$ and $g$, we get $a\in\mathfrak{A}_\phi^{loc}$. Applying Theorem \ref{thm.V.expectation.predictable}, we have that $V$ satisfies the measure integro-differential equation (\ref{eq.V.equation}).
  Note that $\Lambda^{sc}=a^{sc}=0$ in this case.
  Then, following from Theorem \ref{thm.equation.Lebesgue}, we get the conclusion.
\end{proof}

\begin{corollary}\label{cor.UniqueSolution.nonsingular}
  Let $X$ be a general PDMP with $\Lambda^{sc}=0$. Assume that $J_\Lambda$ contains no confluent state.
  Suppose that $f\in\D(\A)$ is piecewise absolutely path-continuous with $J_{Df}\subseteq J_\Lambda\cup\mathrm{Supp}_g$ and satisfies
  \begin{equation}\label{eq.f.eq.nonsingular}
    \left\{\begin{array}{l}
      \displaystyle\mathcal{X}f(x)+\lambda(x)\int_E[f(y)-f(x)]Q(x,\ud y)+l(x)=\delta f(x)\quad \hbox{path-a.e.},\\
      \displaystyle\Delta f(x)+\Delta\Lambda(x)\int_E[f(y)-f(x)]Q(x,\ud y)+g(x)=0,\; x\in J_\Lambda\cup\mathrm{Supp}_g,
    \end{array}\right.
  \end{equation}
  where $l$ is a locally absolutely path-integrable function, and $g$ is a locally absolutely path-summable function.
  Suppose further that $f$ satisfies that
  $$\Ep_x\left[\sum_{n=1}^\infty e^{-\delta\tau_n} \big|f(X_{\tau_n})-f(X_{\tau_n}^-)\big|\right]<\infty,\quad x\in E$$
  and that $\Ep_x[e^{-\delta t}f(X_{t\land\tau})]\to 0$ as $t\to\infty$ for all $x\in E$.
  Then
  \begin{equation}\label{eq.unique.sol.nonsingular}
    f(x)=\Ep_x\left[\int_0^\tau e^{-\delta s}l(X_s)\ud s +\sum_{0<s<\tau}e^{-\delta s} g(X_s^-)\right],\quad x\in E.
  \end{equation}
\end{corollary}

\begin{proof}
  For a general PDMP $X$ with $\Lambda^{sc}=0$, if $f$ is piecewise absolutely path-continuous with $J_{Df}\subseteq J_\Lambda\cup\mathrm{Supp}_g$, then $\A f(x,\cdot)$ is piecewise absolutely continuous for all $x\in E$ with $J_{\A f}\subseteq J_\Lambda\cup\mathrm{Supp}_g$. Thus, it is following from the equations (\ref{eq.f.eq.nonsingular}) that $f$ satisfies the measure integro-differential equation (\ref{eq.f.eq.general}) with
  $$a(x,t)=\int_0^tl(\phi(s,x))\ud s+\sum_{0<s\leqslant t}g(\phi(s,x)),\quad t\in\mathcal{I}_x,\,x\in E.$$
  By Theorem \ref{thm.UniqueSolution}, the result is following from
  $$f(x)=\Ep_x\left[\int_{(0,t]} e^{-\delta s}\ud A_s\right],\quad x\in E.$$
  This completes the proof.
\end{proof}




\bibliographystyle{elsart-num-sort}
\bibliography{Measure-valued_generator_of_general_PDMPs}





\end{document}